\documentclass{article}

\RequirePackage[OT1]{fontenc}
\RequirePackage{amsmath,amsfonts,amsthm,amssymb}
\usepackage{authblk}

\numberwithin{equation}{section}

\newtheorem{theorem}{Theorem}

\newtheorem{lemma}{Lemma}
\newtheorem{proposition}{Proposition}
\newtheorem{remark}{Remark}
\newtheorem{definition}{Definition}

\usepackage{graphicx}
\usepackage[a4paper,top=2cm, bottom=2cm, left=2cm, right=2cm]{geometry}

\newcommand{\bydef}{\stackrel{\rm{def}}{=}}

\newcommand{\mI}{{\mathcal I}}
\newcommand{\mJ}{{\mathcal J}}
\newcommand{\mS}{{\mathcal S}}

\newcommand{\mM}{{\mathcal M}}
\newcommand{\mL}{{\mathcal L}}
\newcommand{\aL}{L} 
\newcommand{\dL}{\Lambda} 

\newcommand{\mX}{{\mathcal X}}
\newcommand{\bN}{{\mathbb N}}
\newcommand{\bR}{{\mathbb R}}
\newcommand{\bZ}{{\mathbb Z}}
\newcommand{\bP}{{\mathbb P}}
\newcommand{\bE}{{\mathbb E}}

\begin{document}
%
%
%
%
 \title{Closed queueing networks under congestion: non-bottleneck independence and bottleneck convergence}
%
\author[1]{Jonatha Anselmi}
\affil[1]{Basque Center for Applied Mathematics (BCAM), anselmi@bcamath.org}
%
\author[2]{Bernardo D'Auria}
\affil[2]{Universidad Carlos III de Madrid, bernardo.dauria@uc3m.es}
%
\author[3]{Neil Walton}
\affil[3]{University of Amsterdam, n.s.walton@uva.nl}
\date{}
%
%


\maketitle

\abstract{
We analyze the behavior of closed multi-class product-form queueing networks when the number of customers grows
to infinity and remains proportionate on each route (or class).
First, we focus on the stationary behavior and prove the conjecture that the stationary distribution
at non-bottleneck queues converges weakly to the stationary distribution of an ergodic, open
product-form queueing network, which is geometric. This open network is obtained by replacing bottleneck queues with
per-route Poissonian sources whose rates are uniquely determined by the solution of a strictly concave
optimization problem. We strengthen such result by also proving convergence of the first moment of the queue lengths of
non-bottleneck stations.
Then, we focus on the transient behavior of the network and use fluid limits to prove that the
amount of fluid, or customers, on each route eventually concentrates on the
bottleneck queues only, and that the long-term proportions of fluid in each route and in each queue
solve the dual of the concave optimization problem that determines the throughputs of the previous
open network.
}

\section{Introduction}

Complex systems such as communication and computer networks are composed of a number of interacting
particles (or customers) that exhibit important congestion phenomena as their level of interaction
grows. The dynamics of such systems are affected by the randomness of their underlying events, e.g.,
arrivals of units of work, and can be described stochastically in terms of queueing network models.
Provided that these are tractable, they allow one to make predictions on the performance achievable
by the system, to optimize the network configuration and to perform capacity-planning studies. These
objectives are usually difficult to achieve without a mathematical model because real systems are
huge in size; e.g., Urgaonkar et al. \cite{Urgaonkar05}.

The focus of this paper is on the well-known class of \emph{closed} queueing network models introduced in
Kelly \cite{Kelly79} and Baskett et al. \cite{BasCMP75}. Specifically, a fixed number of customers
circulate in a network following given \emph{routes}. A route is a sequence of queues (or stations)
that forms a cycle in the network.
In terms of the amount of service required on each queue, users belonging to the same route are
statistically equivalent. In contrast, users belonging to different routes can be statistically
different.
The stationary probability distribution of
these queueing networks has the \emph{product-form} property, which formally means that it can be written
as the product of simple terms associated to each queue up to a normalizing constant.
This surprising property represents a big step forward for the understanding of the stationary
behavior of this queueing network model, as it drastically reduces the intrinsic computational complexity of solving
the global balance equations of the underlying Markov chain. However, due to the quick growth of the state space and despite
the attention devoted to this problem during the last decades, the computation of the normalizing
constant remains a notoriously difficult task, especially when the number of customers is large.
This issue limits the application of these models to optimization and dimensioning studies of real
systems.

During the last decades, several approaches have been investigated to assess the stationary behavior
of closed product-form queueing networks with a large number of customers.
A large body of the literature aims at developing exact algorithms for the efficient computation of
the normalizing constant or stationary performance indices such as mean queue lengths and throughputs;
see, e.g., Reiser and Kobayashi \cite{ReiK75b}, Reiser and Lavenberg \cite{RL80}, Harrison and Coury
\cite{Harr02}, Casale \cite{Casale2011} and the references therein. To the best of our knowledge, no
exact algorithm has a running time that is polynomial with respect to the numbers of routes,
customers and queues.
Motivated by this difficulty, a number of alternative analyses emerged in the literature for the
stationary behavior.
These mainly consist in:
\begin{itemize}
\item Using the mean value analysis (MVA) by Reiser and Lavenberg \cite{RL80} to develop iterative
or fixed-point algorithms; e.g., Schweitzer \cite{Scw}, Chandy and Neuse \cite{Chandy82}, Pattipati
et al. \cite{Pattipati90}, Wang et al. \cite{Wang08}. While these techniques improve the running
time of MVA, there is no guarantee that they converge to the exact solution except for particular
cases.
\item Developing efficient bottleneck identification techniques; e.g., Schweitzer \cite{Scw81},
Schweitzer et al. \cite{SSB93}, Casale and Serazzi \cite{Casale04}, Anselmi and Cremonesi
\cite{Anselmi2010}. This approach aims at reducing the network size by ignoring the impact of the
stations that have a minor influence on the overall performance.
\item Studying the network behavior when some parameter approaches a limiting value of practical interest, which is the approach followed in this article.
\end{itemize}

A number of limiting regimes have been studied in queueing networks.
For the case of single-class networks, see Goodman and Massey \cite{GoMa84}. More generally, through a number of examples, Whitt \cite{Whitt84} expounds the technique of approximating
closed systems with the behaviour of an appropriate open queueing system. For open and closed migration processes,  Pollett \cite{Po00} and Brown and Pollett \cite{BrPo82} find methods to identify bottlenecks and, in addition, characterize the total variation distance between the output process of some queue and a Poisson process. 
Works analyse the fluid and heavy-traffic limit of single-class closed queueing networks. Chen and Mandelbaum \cite{ChMa91} provides a heavy traffic approximation of closed networks, and a characterization of bottleneck behaviour is found.
The paper extends these results to two classes of customers with different priorities. In Kaspi and Mandelbaum \cite{KaMa92}, the
generality of the service type distributions is greatly increased and a concrete
connection between the limiting behaviour of bottleneck queues and the stationary
distributions of a Brownian network is made. 
The book of Chen and Yao \cite{ChenYao2001} analyzes the fluid and diffusion limits for generalized Jackson networks, for both open and closed topologies. 
The word generalized refers to the fact that the external Poisson arrivals are substituted by general and independent renewal processes. The analysis there, though, mainly considers single-class networks,
or in the multi-class case focuses on the FIFO discipline. This discipline is very different from the processor sharing one and shows qualitatively different behavior, for instance unexpected instability occurs when all stations of the network are individually stable for the solution of the traffic flow equations; Kumar and Seidman \cite{Kumar90dynamicinstabilities}. Fluid analysis of multi-class open queueing networks is found in Bramson \cite{Br96}. Detailed discussions on Bramson's work will be made in subsequent sections of this article. 

In the closed multi-class queueing networks investigated in this article, a number of works have
assumed the existence of an $M/M/\infty$ queue and allowed the total number of jobs,
say $n$, grow in proportion to its service (also known as `think') times; see e.g. McKenna and Mitra \cite{MitraMcKenna86},
McKenna \cite{McKenna84}, Berger et al. \cite{Berger99}. 
Further analyses let~$n$ grow to infinity in proportion to the number of stations; Knessl and Tier \cite{Knessl90,Knessl92}.
Finally, another important approach is to study the network behavior when~$n$ grows to infinity keeping fixed the
proportions of customers in each route; Pittel \cite{Pittel}, Schweitzer \cite{Scw81}, Balbo and
Serazzi \cite{BalS97}, Walton \cite{Walton09}, Walton et al. \cite{Walton09b}, Anselmi and Cremonesi
\cite{Anselmi2010}, George et al. \cite{Xia12}.

The focus of this paper is on the last approach described above where $n$, the population vector, grows to infinity keeping fixed the proportions of
customers in each route.
In this limiting regime, it is known that some queues, called \emph{bottlenecks} in the following,
increase their backlog proportionally to~$n$, see Pittel \cite{Pittel} and Walton et al.
\cite{Walton09b}, and uniquely determine the throughput of customers along each route by means of a
concave optimization problem; Schweitzer \cite{Scw81}, Walton \cite{Walton09}.
Interestingly, this optimization problem coincides with the utility optimization problem that
determines the fractions of bandwidth (or rates) allocated to multiple classes of concurrent
Internet flows (or end-to-end Internet connections); see Kelly et al. \cite{KMT98}, Srikant
\cite{Sr04}.
On the other hand, the amount of backlog in each \emph{non}-bottleneck is strictly bounded by~$O(n)$
but in general its limiting behavior is not known. There is numerical evidence to support the
conjecture that the limiting stationary distribution of each non-bottleneck queue is geometric;
Balbo and Serazzi \cite{BalS97}. Such convergence was proved by Pittel \cite{Pittel} under the
assumption that the mode of the stationary distribution is unique; see also Gordon and Newell
\cite{GorNew67}, Lipsky et al. \cite{Lipsky82}, Anselmi and Cremonesi \cite{Anselmi2010}. This
assumption, for instance, is not satisfied when the number of bottlenecks is less than the number of
routes but greater than one, which is a natural scenario in real networks.
As we shall explain, one of the principal contributions of this paper is to prove this conjecture by
removing the unimodal assumption.

The above conjecture provides an accurate and efficient approximation to mean performance indices that takes into account the contribution of
non-bottlenecks.
The impact on performance of non-bottlenecks becomes
non-negligible if their number prevails significantly, and this is often the case; see Casale and
Serazzi \cite{Casale04}, Anselmi and Cremonesi \cite{Anselmi2010}.
This approximation allows for the direct development of efficient optimization frameworks able to
address, for instance, data-center consolidation problems \cite{consolidation08}, where the objective is to reduce the cost
and the size of a data-center while guaranteeing a given performance level. Furthermore, it provides
a good initial guess for the iterative or fixed-point algorithms mentioned above.

\subsection{Our contribution.}

Following the approach considered by a large body of the literature, we are interested in the
behavior of closed product-form queueing networks when the number of customers in each route grows
to infinity proportionally. This is mainly motivated because real networks are populated by a large
number of customers; e.g., Urgaonkar et al. \cite{Urgaonkar05}.
Our objective is two-fold and consists in analyzing the problem from two contrasting standpoints.

The first part of this paper focuses on the \emph{stationary} behavior of these networks. We prove
the conjecture that the stationary distribution of non-bottlenecks converges weakly to the
stationary distribution of an ergodic, open product-form queueing network, with geometrically distributed queue lengths. 
This open network is
obtained by replacing bottlenecks with per-route Poissonian sources whose rates are uniquely determined by
the solution of a strictly concave optimization problem.
In addition, we strengthen such result by proving that the mean per-route number of customers of 
non-bottlenecks converges to the corresponding mean number of customers of the same open network, which is important from an
operational standpoint.

On the other hand, the second part of this paper focuses on the \emph{transient}
behavior of the fluid limit version of the network. 
We start from any arbitrary distribution of customers in the network that preserves the proportional distribution. 
Then, we let their total number go to infinity. This allows, after a
renormalization and a limit procedure, to construct the fluid limit version of the network with 
a given initial concentration of fluid in its queues.
We prove that the amount of fluid, or customers, on each route eventually concentrates, as time increases, on the bottleneck queues only and that the (long-term) proportions of fluid in each route and in each bottleneck solve the dual of the concave optimization problem that determines the throughputs of the open network described in the first part above.
Our proof for closed queueing networks uses an entropy Lyapunov function similar to the one used by
Bramson \cite{Br96} to establish convergence properties of the fluid-limit equations of open
queueing networks.

The technical difference behind the two results above is the order in which the limit in the
number of customers and the limit in time are taken. In the second part, the limit in the number of
customers is taken \emph{before} the limit in time, and vice versa.
In stochastic systems, these two limits do not commute in general,
but for the class of queueing networks investigated in this paper we prove that they do.
Taking the limit in the number of customers first provides a natural way to look at the
evolution of a network populated by a large number of customers and, by subsequently taking a limit
in time, we justify fluid model arguments within the queueing literature.
The second result proven in this paper, thus,  increases the robustness
of the approach taken in the first part, which has been followed by several
researchers as referenced above. Furthermore, it can be also seen as a queueing theoretic
analysis of the utility optimization found in congestion control protocols; e.g., Kelly et al.
\cite{KMT98} and Srikant \cite{Sr04}.

\subsection{Organization.} 
In Section~\ref{model}, we introduce the model considered in this paper. In particular,
we include two Markov descriptions of a closed queueing network, relevant quantities such as
bottleneck and non-bottleneck queues are defined, and an expression
for a fluid model of a closed queueing network is given. In Section~\ref{results}, we present the
three main
results of this paper: Section~\ref{result1} shows the asymptotic independence of non-bottleneck
queues 
in the large-population limit; Section~\ref{result2} shows the convergence of the
Markov closed queueing network to a fluid solution; and Section \ref{result3} states that a
fluid solution converges to the set of bottlenecks in a way that minimizes a certain entropy
Lyapunov function. In Section~\ref{proofs}, we prove the main results stated in
Section~\ref{results}. We respectively prove the results of Sections~\ref{result1}, \ref{result2}
and \ref{result3} in Sections~\ref{proof1}, \ref{proof2} and \ref{proof3}.

\section{Closed queueing networks models.}

\label{model}

We consider closed, multi-class queueing networks in the sense of Kelly \cite{Kelly79} and Baskett
et al. \cite{BasCMP75}. 
The set~$\mJ\subset\bN$ denotes the set of queues (or stations) and we let $J=|\mJ|$.
The set~$\mI\subset 2^{\mJ}$ denotes the set of routes (or classes) and we let $I=|\mI|$. A
route is a sequence of queues visited by a customer during one cycle of the
network. We assume, for simplicity, that each customer visits each queue at most once within a
cycle of the network. Within each route $i=\{j^i_1,...,j^i_{k_i}\}$, we associate a route order
$(j^i_1,...,j^i_{k_i})$. For $k=1,...,k_i-1$, a customer departing queue $j^i_k$ will next join
queue $j^i_{k+1}$ and a customer departing queue $j^i_{k_i}$ will join queue queue $j^i_1$.
Unless otherwise specified, $i$ will be used to index routes and $j$ will be used to index queues.
We assume that a constant number of customers circulate along each route of the network. We denote
by
$n=(n_i:i\in\mI)\in\bN^I$ the \emph{population vector}, i.e., the total number of customers on each
route. When joining queue~$j$, we assume that route-$i$ customers require amounts of service that
are independent and exponentially distributed
with mean $\mu^{-1}_{ji}$. At each queue, we assume that customers
are served at rate $1$ according to a processor sharing discipline and customers joining a queue
take a position uniformly at random in the queue.
Thus, if a route-$i$ customer does not join a queue $j$, i.e. $j\notin i$, we may assume
$\mu^{-1}_{ji}=0$.


\subsection{Two Markov models of closed queueing networks.}

Firstly, we could describe the exact location of each customer in
the queue according to its route type. Here the \emph{explicit state} of a queue $j$ would be a
vector $s_j=(i_j(k): k=1,...,m_j)$, where $m_j$ is the number of customers in the queue and $i_j(k)$
is the route type of the customer in the $k^{th}$ position. The explicit state of the network would
then be the vector of each queue's state, $s=(s_j: j\in\mJ)$. We then let $\mX(n)$ be the set
of the explicit states where the number of customers of each route type $i$ is $n_i$.

Secondly, we could ignore positional information about customers within a queue, and instead,
just consider the number of each route type at a queue. Here, we let $m=(m_{ji}: j\in\mJ, i\in\mI,
j\in i)$ be a network \emph{state}, where $m_{ji}$ represents the number of route $i$ customers in
queue~$j$. 
Thus,~$\mS(n)=\{m:\sum_j m_{ji}=n_i, i\in\mI\}$ is the state space of this Markov chain. It may be
verified in a straightforward manner that this state space is finite and irreducible.

Under the above assumptions, it is known (Kelly \cite{Kelly79} and Baskett et al. \cite{BasCMP75})
that the stationary distribution of
being in state~$s$, which we denote by $\pi(s|n)$, or in state $m$, which we denote
$\pi(m|n)$,
is respectively
\begin{align}
\pi(s|n) &=\frac{1}{B(n)}\prod_{j\in\mJ} \prod_{k=1}^{m_j}
\mu_{ji_j(k)}^{-1}=\frac{1}{B(n)}\prod_{j\in\mJ}\prod_{i:j\in i}\mu_{ji}^{-m_{ji}}, 
&\hspace{-2cm} s\in\mX(n), \label{eq:piexplicit}\\
\pi(m|n) &=\frac{1}{B(n)}\prod_{j\in\mJ} \left(\binom{m_j}{m_{ji}:i\ni j} \prod_{i: j\in
i}\mu_{ji}^{-m_{ji}}\right),
&\hspace{-2cm} m\in\mS(n), \label{eq:pi} 
\end{align}
where $B(n)$ is the normalizing constant
\begin{equation}\label{Bn}
B(n)\bydef\sum_{m\in\mS(n)} \prod_{j\in\mJ} \left(\binom{m_j}{m_{ji}:i\ni j} \prod_{i:
j\in
i}\mu_{ji}^{-m_{ji}}\right) \ ,
\end{equation}
that is the same in (\ref{eq:piexplicit}) and (\ref{eq:pi}); see \cite{BasCMP75}. 
To count the possible orderings of customers inside a queue, we use the multinomial coefficient
\begin{equation*}
\binom{m_j}{m_{ji}:i\ni j}\bydef\frac{m_j!}{\prod_{i:j\in i} m_{ji}!}.
\end{equation*}

In the following, we mainly consider the stationary distribution~\eqref{eq:pi}, while the
expression~\eqref{eq:piexplicit} will be used in proofs.

\begin{remark}
There are a large number of
generalizations of our queueing network where \eqref{eq:pi} still gives the stationary probability of the per-route number of
customers in each queue.
For instance, we could
generalize our processor-sharing discipline to the class of symmetric queueing disciplines with unit
service capacity, see Kelly \cite[Section 3.3]{Kelly79}. We could also generalize service
requirements to be independent with mean $\mu^{-1}_{ji}$ and having a rational Laplace transform.
If we keep the assumption that service requirements are exponential and we
assume that the state space of our Markov description is irreducible, then we could generalize to
allow any service discipline that allocates service amongst customers in a way that
does
not discriminate between the route types of customers at the queue. Such service disciplines are
described by Kelly \cite[Section 3.1]{Kelly79}. Finally, in all cases, we may elaborate the routing
structure of our network; we may allow a route $i$ customer to be routed through the network as a
Markov chain (with finite expected exit time) whose states are determined by the set of queues so
far visited.
\end{remark}

Although generalizations leading to stationary distribution \eqref{eq:pi} are
abundant (as explained in the above remark), we do not explore these in further detail for simplicity of exposition. 
Thus, our results in Theorems~\ref{Thrm1} and~\ref{Thrm1_expectation} generalize to such cases.
On the other hand, Theorems~\ref{main fluid theorem} and~\ref{Thrm3} apply to the specific case of processor-sharing queues, and thus they do not apply directly to all the generalizations described above.
The extension of Theorems~\ref{main fluid theorem} and~\ref{Thrm3} when other service disciplines are considered, e.g., last-come-first-served preemptive-resume, is left as future work.

\subsection{Throughputs and bottlenecks.}

A key quantity of interest is the rate at which customers complete service on each route with
respect to a reference queue (say $j_1^i$ for route $i$). Let
\begin{equation}\label{thoughput expression}
\Lambda_i(n)\bydef \sum_{\substack{m\in\mS(n):\\ m_{j^i_1}>0}} \mu_{j^i_1
i} \frac{m_{j^i_1 i}}{m_{j^i_1}} \, \pi(m|n)
\end{equation}
be the \emph{throughput} of customers on route $i$ observed at queue $j^i_1$. In addition,
let
\begin{equation}
U_j(n)\bydef\sum_{i: j\in i} \frac{\Lambda_i(n)}{\mu_{ji}}
\end{equation}
be the \emph{utilization} (or load) of queue~$j$. Using Little's law, $U_j(n)$ can be interpreted as the proportion of time where queue~$j$ is busy.
\begin{remark}
The expression \eqref{thoughput expression} gives the service completion rate of route $i$
customers at queue $j^i_1$, i.e., $\mu_{j^i_1i}$, times the service rate devoted to these customers
at that queue, i.e., ${m_{j^i_1 i}}/{m_{j^i_1}}$, times the stationary probability that there are
$m$ customers in the network, i.e., $\pi(m|n)$. Thus, this expression gives the mean rate (or
throughput) for which route $i$ customers leave queue~$j^i_1$.
Note the route $i$ throughputs at each queue $j\in i$ must be equal. Without loss of generality, we
considered queue $j^i_1$.
\end{remark}

A more concise expression for the per-route throughput is given in the following lemma; 
see, e.g., Bolch et al. \cite[Formula~(8.28)]{Bolch} or, for single class networks, 
Chen and Yao \cite[Formula~(2.10)]{ChenYao2001}.
\begin{lemma}\label{simple Lambda}
\begin{equation}\label{throughput.formula}
\Lambda_i(n)=\frac{B(n-e_i)}{B(n)},\qquad i\in\mI
\end{equation}
where $B(n)$ is the normalizing constant \eqref{Bn} and $e_i$ is the $i^{th}$ unit vector in
$\bR_+^I$. 
\end{lemma}

%

In agreement with other works (e.g., Anselmi and Cremonesi \cite{Anselmi2010}, Balbo and Serazzi
\cite{BalS97}), we define a bottleneck queue as a queue whose service is saturated.

\begin{definition}
\label{eq:bdef}
Queue $j\in\mJ$ is called \textit{bottleneck} if and only if
\begin{equation}
\label{eq:bottleneck_def}
\lim_{c\rightarrow\infty} U_j(cn+u_c)=1,
\end{equation}
where $\{u_c\}_{c\in\bN}$ is any uniformly bounded sequence such that $c n + u_c\in\bZ_+^I$.
We define the set $\bar{\mJ}\subseteq \mJ $ to be the set of bottlenecks and let
$\bar{J}=|\bar{\mJ}|$. Similarly, we define $\mJ^\circ=\mJ\backslash \bar{\mJ}$ to be the set of
non-bottlenecks and let $J^\circ=|\mJ^\circ|$. 
\end{definition}

We will think of $\mJ^\circ$ as the `open part' of the closed network under investigation.

\subsection{Non-bottleneck queues and open queueing networks.}

We are interested in the probability distribution of non-bottleneck queues. For this reason, we
consider
queue-size vectors $m^\circ=(m^\circ_{ji}: j\in\mJ^\circ, i\in\mI, j\in i)\in\bZ_+^{J^{\circ}\times
I}$ and
\begin{align}\label{closed nb ed}
& \pi^\circ(m^\circ|n) \bydef
\sum_{\substack{\bar{m}\in\bZ_+^{\bar{\mJ}}:\\(\bar{m},m^\circ)\in\mS(n)}}
\pi((\bar{m},m^\circ)|n),\qquad j\in\mJ^\circ, i\in\mI,
\end{align}
which defines the stationary probability that non-bottleneck queues are in state $m^\circ$. 

We also define
\begin{subequations}\label{open ed}
\begin{equation}
\label{open ed a}
\pi^\circ_\Lambda(m^\circ)\bydef\prod_{j\in\mJ^\circ} \pi^\circ_{j,\Lambda}(m^\circ),
\end{equation}
$m^\circ\in\bZ_+^{J^{\circ}\times I},\;\Lambda\in\bR_+^{I}$, where
\begin{equation}
 \pi^\circ_{j,\Lambda}(m^\circ) \bydef \bigg(1- \sum_{i:j\in i}
\frac{\Lambda_i}{\mu_{ji}}\bigg)\left(\begin{array}{
cc } m_j\\m_{ji}\;:i\ni j \end{array}\right)\prod_{i: j\in
i}\left(\frac{\Lambda_i}{\mu_{ji}}\right)^{m_{ji}}.
\end{equation}
\end{subequations}
The distribution $\pi$, given by \eqref{eq:pi}, refers to the stationary distribution of a closed
queueing network.
In contrast, the distribution $\pi^\circ_\Lambda$ can be shown to be the stationary
distribution of an \emph{open} queueing network constructed on queues $\mJ^\circ$. Here customers
arrive on
each route as a Poisson process of rate $\Lambda=(\Lambda_i: i\in\mI)$ and depart the network after
receiving service at each queue on their route $i\cap \mJ^\circ$, see Kelly \cite{Kelly79} and
Baskett et al. \cite{BasCMP75}.

\subsection{Fluid Model.}
In order to study the transient behavior of our closed queueing network, we will analyze the
following fluid model.

\begin{definition}[Closed queueing network fluid model]\label{def:fluid.sol}
\label{def_fluid}
The processes $m(t)=(m_{ji}(t): j\in\mJ, i\in\mI, j\in i)$ and $\aL(t)=(\aL_{ji}(t):
j\in\mJ, i\in\mI, j\in i)$
form a \emph{fluid solution} (or fluid limit) of our closed queueing network if they satisfy the
following conditions:
\begin{subequations}
\label{fluidequns}
 \begin{align}
  m_{j^i_k i}(t)&=\aL_{j^i_{k-1} i}(t)- \aL_{j^i_k i}(t),\label{fluidequns1}\\
\sum_{i: j\in i} \frac{1}{\mu_{ji}}&\left( \aL_{ji}(t)-\aL_{ji}(s) \right) \leq
t-s,\label{fluidequns2}\\
\aL_{ji}(t)&\; \text{is increasing,}\label{fluidequns3}\\  
\text{if } m_{j}(t)>0 &\text{ then }
\dL_{ji}(t)=\frac{m_{ji}}{m_j}\mu_{ji},\label{fluidequns4}\\
\sum_{j: j\in i} &m_{ji}(t)=n_i.\label{fluidequns5}
 \end{align}
\end{subequations}
where we used the definition  $\dL_{ji}(t) \bydef d\aL_{ji}(t)/dt$.
Here, $j\in\mJ, i\in\mI, j\in i$, $t\geq s \geq 0$. Also, $j^i_k$ is the $k^{th}$ queue on route $i$
and
$j_{k-1}^i$ is the queue before the $k^{th}$ queue (we use the convention that the queue
before $j^i_1$ is $j^i_{k_i}$).
\end{definition}

The $(j,i)$-component of the process $m(t)$ denotes the amount of fluid of type $i$
contained at time $t$ at queue $j\in i$. The process $\dL_{ji}(t)$ gives the instantaneous
throughput of this type of fluid at the same queue, while the process $\aL_{ji}(t)$ gives
the total amount of fluid of this type flowed out from queue $j$ by time $t$.

The conditions~\eqref{fluidequns}, thus, are the defining properties of a fluid solution and we
observe that they are analogous to the ones used by Bramson \cite[Formulas (2.3)-(2.6)]{Br96}, which
hold
for the open versions of the considered closed queueing networks\footnote{Actually, the open
queueing networks defined in Bramson \cite{Br96}, called head-of-the-line processor sharing
networks, appear
different from ours. In fact, upon completion of service at one queue, a class-$i_1$ customer
becomes of class $i_2$ with probability $p_{i_1i_2}$. However, one can easily build a mapping from
one representation to the other.}.
In particular, conditions~\eqref{fluidequns1}, \eqref{fluidequns3}, \eqref{fluidequns5} are basic
and relate queue lengths in an obvious manner,
condition~\eqref{fluidequns4} is the property that defines a processor-sharing discipline, and
condition~\eqref{fluidequns2} takes into account the maximal processing rate of the system.

We note that the condition~\eqref{fluidequns2} implies that $\aL_{ji}$
is Lipschitz continuous. By~\eqref{fluidequns1}, this is also true for $m_{ji}$. Lipschitz
continuity implies
absolute continuity, and therefore the processes $\aL_{ji}$ and $m_{ji}$ are differentiable
almost everywhere with respect to the Lebesgue measure. Throughout this document, the term
\emph{for almost every} will refer to a set of real numbers whose complement has Lebesgue measure
zero. Shortly, we will prove that the limit of the closed queueing network described previously
satisfies the fluid model~\eqref{fluidequns}.

\section{Results on closed queueing networks.}\label{results}
In this section, we present the main results of this article. Namely,
Theorems~\ref{Thrm1} and~\ref{Thrm1_expectation}, which state the independence of the per-route numbers of customers in non-bottlenecks and the convergence of their expectations under a large-population limit (Section~\ref{result1}); 
Theorem~\ref{main fluid theorem}, which states that the stochastic process limit of a closed queueing network is a solution to the fluid equations \eqref{fluidequns} (Section~\ref{result2});
Theorem \ref{Thrm3}, which states the convergence in time of the fluid model \eqref{fluidequns} via a Lyapunov function argument (Section~\ref{result3}).

\subsection{Independence of non-bottleneck queues.}
\label{result1}
We will demonstrate that non-bottleneck queues become independent as our closed queueing network becomes congested.
We now consider the limiting behavior of our queueing network with $cn$ customers when
$c\rightarrow\infty$. Given the constraints on the number of customers on each route, the
random variables of the stationary number of per-route customers in each queue are certainly dependent; however, in classic open
queueing networks, where all customers arrive from an external source and eventually leaves the network, they are shown to be independent, see Kelly \cite{Kelly79} and Baskett et al. \cite{BasCMP75}. 

We consider our stationary distributions for closed multiclass queueing networks, \eqref{closed nb ed}, and open multiclass queueing networks, \eqref{open ed}.
In informal terms,  we wish to prove that for each $n\in\bR_+^I, m\in\bZ_+^{J^\circ \times I}$ 
\begin{equation}\label{informal statement}
\pi^\circ(m^\circ| cn) \xrightarrow[c\rightarrow\infty]{}
\pi_{\Lambda^*}^\circ(m^\circ).
\end{equation}
To prove such a statement, we must identify the throughput vector $\Lambda^*\in\bR_+^I$ and the set
of non-bottleneck queues $\mJ^\circ$. Since the vector $cn$ need not to belong to
$\bZ_+^I$, we consider $c n+u_c$ where $\{u_c\}_{c\in\bN}$ is any uniformly bounded sequence such
that $c n + u_c\in\bZ_+^I$. Our theorem can then be written as follows.
\begin{theorem}\label{Thrm1} For $n\in\bR_+^I, m^\circ\in\bZ_+^{J^\circ \times I}$ 
\begin{equation}\label{formal statement}
 \pi^\circ(m^\circ| c n+u_c) \xrightarrow[c\rightarrow\infty]{}
\pi_{\Lambda^{*}(n)}^\circ(m^\circ)
\end{equation}
where $\Lambda^{*}(n)=(\Lambda_{1}^{*}(n),\ldots,\Lambda_{I}^{*}(n))$ is the unique optimizer of the
strictly-concave optimization problem
\begin{subequations}\label{PF Opt}
 \begin{align}
\text{maximize} \quad & \sum_{i\in\mI} n_i \log \Lambda_i\label{PF Opt 1}\\
\text{subject to} \quad & \sum_{i:j\in i} \Lambda_i / \mu_{ji} \le 1,& j\in\mJ \label{PF Opt 2}\\
\text{over} \quad\qquad & \Lambda_i\geq 0,& i\in\mI,\label{PF Opt 3}
\end{align} 
\end{subequations}
and where $\mJ^\circ$, the set of non-bottlenecks, is given by the set of queues $j$ such that
\begin{equation*}
 \sum_{i: j\in i} \frac{\Lambda_{i}^{*}(n)}{\mu_{ji}} < 1.
\end{equation*}
\end{theorem}

The solution of optimization problem \eqref{PF Opt}, $\Lambda^*(n)$, is known in the literature as
the \emph{proportionally-fair} allocation (see Kelly \cite{Ke97}) and, interestingly, emerged
independently as a model for the sharing of bandwidth among Internet connections (e.g., Srikant
\cite{Sr04}).
For a detailed treatment of the relationship between closed queueing networks and the
proportionally-fair allocation, see Walton \cite{Walton09}.

Theorem 1 proves that non-bottleneck queue lengths converge in distribution.  Unfortunately, this is not enough to claim that the first moments converge as well. Convergence of the first moments is important from a practical standpoint, particularly for mean value analyses. This result is demonstrated in the next theorem. 
For our closed queueing network with $n$ customers, let $M^\circ(n)\in\bZ_+^{J^\circ \times I}$ be the stationary number of customers from each route at each non-bottleneck queue. In other words, $\mbox{Pr}(M^\circ(n)=m^\circ)=\pi^\circ(m^\circ|n)$.
%
Similarly, let $M_{\Lambda}^\circ$ be the stationary number of customers from each route at each queue in the open queueing network \eqref{open ed a}, i.e., $\mbox{Pr}(M_{\Lambda}^\circ = m^\circ) =\pi^\circ_\Lambda(m^\circ)$.

\begin{theorem}\label{Thrm1_expectation}
\begin{equation}
E M^\circ(c n + u_c)   \xrightarrow[c\rightarrow\infty]{}  E M^\circ_{\Lambda^{*}(n)}
\end{equation}
where $\Lambda^{*}(n)=(\Lambda_{1}^{*}(n),\ldots,\Lambda_{I}^{*}(n))$ is the unique optimizer of \eqref{PF Opt}.
\end{theorem}

\subsection{Existence of fluid limits for closed queueing networks.}\label{result2}
\label{sec:fluid_limit}

In Section~\ref{result1}, we analyzed the \emph{stationary} probability distribution \eqref{eq:pi} of the
closed queueing networks under investigation in the large-population limit.
Now, we focus on the \emph{transient} probability distribution in the large-population limit and
then study the evolution in time of the system.
In other words, the limit in time is now taken after the limit in the number of customers.
In stochastic systems it is known that both limits are not interchangeable in general.
The fluid limit, see Definition~\ref{def_fluid}, is a natural framework that allows for the analysis of such scenario.

We consider a sequence of the closed queueing networks as described in Section~\ref{model}.
In this sequence, the only variables that change are the number of customers on each route (the
number of queues, routes, service distributions are kept fixed).
We let the vector $n\in\bR^I_+$ be the proportion of customers on
each route of the network. In the $c^{th}$ network of this sequence of closed queueing networks, there
are $cn+u_c$ customers on each route, where $\{u_c\}_{c\in\bN}$ is a bounded sequence of variables in
$\bR_+^I$ such that $cn+u_c\in\bZ_+^I$.
  
We let $M^c_{ji}(t)$ be the number of route-$i$ customers in queue $j$ at time $t$ of the $c^{th}$
closed queueing network. We let $\aL^c_{ji}(t)$ give the total number of route-$i$ customers
served by queue $j$ by time $t$ in the $c^{th}$ closed queueing network. From this, we define the
rescaled
processes
\begin{equation*}
 \bar{M}^c_{ji}(t)\bydef\frac{M^c_{ji}(ct)}{c}, \qquad \bar{\aL}^c_{ji}(t)\bydef
\frac{\aL^c_{ji}(ct)}{c},\qquad j\in\mJ,\; i\in\mI,\; j \in i.  
\end{equation*}

We wish to show that the vector processes $\bar{M}^c$ and $\bar{\aL}^c$ converge to a fluid
solution.

\begin{theorem}\label{main fluid theorem}
The sequence of stochastic processes $\{ (\bar{M}^c,\bar{\aL}^c\}_{c\in\bN}$ converges weakly
to a continuous process $(m,\aL)$ that satisfies the fluid solution equations \eqref{fluidequns}.
\end{theorem}

\subsection{Convergence of the fluid solution.}\label{result3}

Having established the existence of a fluid solution of our closed queueing
networks, $(m,\aL)$,  now our goal is to study its evolution in the long-term, i.e., $m(t)$ when $t\to\infty$.
In our main result, we show that the amount of fluid $m$ eventually concentrates on the bottleneck
queues only and that the long-term proportions of fluid in each route and in each queue solve the
dual of optimization problem \eqref{PF Opt}.

From the stationary distribution $\pi(m|n)$, given by \eqref{eq:pi}, and under the same premise of
Theorem~\ref{Thrm1}, we can show that
\begin{equation*}
\lim_{c\rightarrow\infty} \frac{1}{c}\log \pi(cm|cn) = -\beta(m) 
\end{equation*}
where 
\begin{align}
\label{eq:def_beta}
\beta(m) &\bydef \sum_{j\in\mJ} \sum_{\substack{i: j\in i\\ m_{ji}>0 } } m_{ji} \log \frac{m_{ji}\mu_{ji}}{m_{j}}\\
 &= \sum_{\substack{j\in \mJ:\\ m_j>0}} m_{j} \sum_{i: j\in i} p_{ji} \log
\frac{p_{ji}}{\mu^{-1}_{ij}}.\label{betalambda}
\end{align} 
Here, we have used the notation $p_{ji} \bydef m_{ji}/m_j$.

\begin{remark}[Relative Entropy]\label{relative entropy remark}
A key quantity that will be useful in our proofs is the unnormalized relative entropy%
\footnote{This entropy is unnormalized because we do not enforce the condition that its arguments
are probability distributions.}
\begin{equation*}
D(p||q)\bydef\sum_{i} p_i \log \frac{p_i}{q_i},\qquad p,q\in\bR_+^I.
\end{equation*}
We notice that the function $\beta$, given by \eqref{betalambda}, is a linear combination of these
unnormalized relative entropies for each queue. Furthermore, we note that if $\sum_i p_i=\sum_i
q_i$, then we can renormalize and treat $p$ and $q$ as probability distributions.
In this case, one can show that $D(p||q)\geq 0$, using Jensen's inequality, and that $D(p||q)$ is minimized if
$p=q$ where its value is $0$.
\end{remark}

\begin{remark}
It is an interesting observation that the rate of decrease of $\beta(m)$, the (negative) entropy of states,
will be determined by the relative entropy between the rates of service; 
see Proposition~\ref{diff prop} below. Although distinct, this argument is similar to those argued
for Markov processes by Spitzer \cite{spitzer1971random} and more recently by Dupuis and Fischer
\cite{DuFi12}. A further important reference is Bramson~\cite{Br96}, which shows, for the open
versions of the closed queueing networks considered in this paper, that the amount of fluid in
non-bottlenecks converges to zero.

In Bramson \cite{Br96}, the main argument behind the fluid convergence to zero for non-bottleneck
stations relies on \emph{closing} the open network with an additional (artificial) queue whose
service rate equals the overall throughput of the network. All customer routes enter this additional
queue and it is assumed that this queue never empties. This plays a significant role in regulating
traffic to match that of the open network. So although a closed queueing network appears in that
analysis, this additional queue introduces an important loss of generality.
\end{remark}

The function $\beta(m)$ forms a natural candidate for a Lyapunov function. We need to show
that, for $n\in\bR_+^I$ fixed, $\beta(m(t))$ decreases to its minimal value
\begin{equation}\label{beta star}
\beta^*\bydef\text{minimize}\quad \beta(m) \quad \text{subject to}\quad \sum_{j\in i}
m_{ji}=n_i,\quad
i\in\mI.
\end{equation}
We show in Lemma \ref{duality-lemma} that the above optimization problem is the dual of the problem
\eqref{PF Opt}, and considering the set of points attaining this minimal value
\begin{equation}
\mathcal{M}(n) \bydef\text{argmin}\quad \beta(m) \quad \text{subject to}\quad \sum_{j\in i}
m_{ji}=n_i,\quad
i\in\mI,
\end{equation}
we have that when the network is in one of these states the corresponding throughputs are uniquely given by
the $\Lambda^*(n)$ vector given in Theorem \ref{Thrm1}.

The following theorem shows how fluid eventually distributes among network stations.

\begin{theorem}\label{Thrm3}
Let $(m(t), \aL(t))$ be a solution of the fluid model \eqref{fluidequns} with $m(0)$ satisfying
equation \eqref{fluidequns5} for a given vector $n$, then,
\begin{equation*}
 \beta(m(t))\searrow \beta^*,\qquad \text{as}\qquad t\rightarrow\infty
\end{equation*}
and, moreover,
\begin{equation*}
 \min_{m^*\in\mathcal{M}(n)} | m(t)- m^*| \rightarrow 0,\qquad \text{as}\qquad t\rightarrow\infty \,
\end{equation*}
where the fluid point $m^*\in\mathcal{M}(n)$ induces, via equation \eqref{fluidequns4}, the same throughputs 
$\Lambda^*(n)$ given in Theorem \ref{Thrm1}.
\end{theorem}

In the above theorem, the norm $|m|$ is the Euclidean norm in $\bR^{\mI\times \mJ}$.
\section{Proofs of main results.}\label{proofs}
We now focus on proving the main results of this paper, namely, Theorems~\ref{Thrm1} and~\ref{Thrm1_expectation} (Section \ref{proof1}), Theorem~\ref{main fluid theorem} (Section \ref{proof2}) and Theorem~\ref{Thrm3} (Section \ref{proof3}).

\subsection{Analysis of non-bottleneck queues.}\label{proof1}

First, we develop a proof of Theorem \ref{Thrm1}. Recalling the informal statement
(\ref{informal statement}), we need to verify that for each $n\in\bR_+^I,
m\in\bZ_+^{J^\circ \times I}$ 
\begin{equation}\label{informal statement2}
\pi^\circ(m^\circ| cn) \xrightarrow[c\rightarrow\infty]{}
\pi_{\Lambda^*}^\circ(m^\circ),
\end{equation}
for some $\Lambda^*\in\bR_+^I$ for some set
of non-bottleneck queues $\mJ^\circ$.
Before verifying such statement, we must identify the relevant
throughput vector $\Lambda$ and non-bottleneck queues ${J}^\circ$. The following
result, which was proven by Walton \cite{Walton09}, characterizes $\Lambda^*$ and $\mJ^\circ$. 
\begin{proposition}[\cite{Walton09}]\label{conv prop}
\begin{equation*}
 \Lambda_{i}^{*}(n)=\lim_{c\to\infty} \Lambda_{i}(c n + u_c),\qquad i\in\mI
\end{equation*}
where $\Lambda_{i}^{*}(n)$ is the unique minimizer of
the concave optimization problem \eqref{PF Opt} and so that $\Lambda_{i}(c n + u_c)$ is defined on a point in its domain, $\{u_c\}_{c\in\bN}$ is any
bounded sequence in $\bR_+^I$ such that $c n + u_c\in\bZ_+^I$.

Consequently, $j\in {\mJ}^\circ$ if and only if $$\sum_{i: j\in i} \frac{\Lambda_{i}^{*}(n)}{\mu_{ji}} < 1.$$
\end{proposition} 

Now, we show that \eqref{formal
statement} holds. Before proceeding with this proof, we introduce a specific closed queueing
network, which will helps us in the proof. Recall the closed queueing network
defined on set $\mJ$ introduced in Section~\ref{model}. Consider a queueing network defined exactly
as in Section~\ref{model}, except that queues $\mJ^\circ$ are removed. The resulting queueing
network has
states $\bar{m}=(\bar{m}_{ji}: j\in\bar{\mJ}, i\in\mI,
j\in i)\in\bZ_+^{\bar{J}\times I}$; if there are $n\in\bZ_+^I$ customers on each route, this
network has state space $\bar{\mS}(n)=\{\bar{m}:\sum_j \bar{m}_{ji}=n_i, i\in\mI\}$, stationary distribution
\begin{align}
\bar{\pi}(\bar{m}|n) &=\dfrac{1}{\bar{B}(n)}\prod_{j\in\bar{\mJ}} \left(\binom{\bar{m}_j}{\bar{m}_{ji}:j\in i} \prod_{i: j\in
i}\mu_{ji}^{-\bar{m}_{ji}}\right),\qquad \bar{m}\in\bar{\mS}(n),
\end{align}
where
\begin{equation*}
\label{eq:barB}
 \bar{B}(n)=\sum_{\bar{m}\in\bar{\mS}(n)} \prod_{j\in\bar{\mJ}} \left(
\binom{\bar{m}_j}{\bar{m}_{ji}:j\in i} \prod_{i: j\in
i}\mu_{ji}^{-\bar{m}_{ji}}\right),
\end{equation*}
and stationary throughput
\begin{equation*}
\bar{\Lambda}(n)=\frac{\bar{B}(n-e_i)}{\bar{B}(n)}.
\end{equation*}
Proposition \ref{conv prop} holds for this network and, here, we would consider
$\bar{\Lambda}^*(n)$, the solution to the optimization
\begin{equation*}
 \text{maximize} \quad  \sum_{i\in\mI} n_i \log \Lambda_i\quad \text{subject to} \quad 
\sum_{i:j\in i} \Lambda_i / \mu_{ji} \le 1,\quad j\in\bar{\mJ} \quad \text{over} \quad \Lambda_i\geq
0,\quad i\in\mI.
\end{equation*}
In this optimization, all constraints that are not relevant to our solution $\Lambda^*(n)$ are removed.
Thus, it is not surprising that the following lemma holds.
\begin{lemma}\label{rate lemma}
$\bar{\Lambda}^*(n)=\Lambda^*(n).$
\end{lemma}
We prove Lemma \ref{rate lemma} in Appendix \ref{appendix1}. A direct consequence of Lemma
\ref{rate lemma}
and Proposition \ref{conv prop} is the following.
\begin{lemma}\label{2nd closed lemma}
 $$\bar{\Lambda}_i(n)=
\frac{\bar{B}(n-e_i)}{\bar{B}(n)}\xrightarrow[c\rightarrow\infty]{}{\Lambda}^*_i(n),\qquad i\in\mI.
$$
\end{lemma}

We can now proceed to demonstrate that \eqref{formal statement} holds. Let us consider the
equilibrium distribution $\pi^\circ (m^\circ|n)$. We have
\begin{align}
 \pi^\circ (m^\circ|n)&= \sum_{\substack{\bar{m}\in\bZ_+^{\bar{\mJ}}:\\(\bar{m},m^\circ)\in\mS(n)}}
\pi((\bar{m},m^\circ)|n)\notag\\
&= \dfrac{1}{B(n)}\times \prod_{j\in\mJ^\circ} \left(\binom{m^\circ_j}{m^\circ_{ji}:j\in
i} \prod_{i: j\in
i}\mu_{ji}^{-m^\circ_{ji}}\right)\times
\sum_{\substack{\bar{m}\in\bZ_+^{\bar{\mJ}}:\\(\bar{m},m^\circ)\in\mS(n)}}\prod_{j\in\mJ}
\left(\binom{\bar{m}_j}{\bar{m}_{ji}:j\in i} \prod_{i: j\in
i}\mu_{ji}^{-\bar{m}_{ji}}\right)\notag\\
&= \dfrac{1}{B(n)}\times \prod_{j\in\mJ^\circ} \left(\binom{m^\circ_j}{m^\circ_{ji}:j\in
i} \prod_{i: j\in
i}\mu_{ji}^{-m^\circ_{ji}}\right)\times
\sum_{\substack{\bar{m}\in\bar{\mS}(n-n^\circ)}}\prod_{j\in\mJ}
\left(\binom{\bar{m}_j}{\bar{m}_{ji}:j\in i} \prod_{i: j\in
i}\mu_{ji}^{-\bar{m}_{ji}}\right)\notag\\
&= \underbrace{\dfrac{\bar{B}(n)}{B(n)}}_{\text{(c)}}\times
\underbrace{\frac{\bar{B}(n-n^\circ)}{\bar{B}(n)}}_{\text{(b)}} \times
\underbrace{\prod_{j\in\mJ^\circ} \left(\binom{m^\circ_j}{m^\circ_{ji}:j\in i} \prod_{i: j\in
i}\mu_{ji}^{-m^\circ_{ji}}\right)}_{\text{(a)}}.\label{3terms}
\end{align}
Here, $n^\circ=(n^\circ_i:i\in\mI)$ is the number of route-$i$ customers in
non-bottleneck queues, i.e. $n^\circ_i=\sum_{j\in i\cap \mJ^\circ} m^\circ_{ji}$. 
The third equality above follows by observing that the summation is over all the states where
the number of non-bottleneck customers is $n-n^\circ$, whch then gives $B(n-n^\circ)$.

We now consider how the terms (a), (b) and (c), converge as we keep
$m^\circ$ fixed and let $n$ increase.
Term (a) is easily dealt with as it does not depend on~$n$.
Term (b) will be shown to converge in the next proposition.
Subsequently, term (c) will take a more in depth analysis.

Term (a) represents the correct expression for the unnormalized stationary distribution of our open queueing network (see \eqref{open ed}),
except that we do not include the multiplicative term
\begin{equation*}
 \prod_{j\in\mJ^\circ} \prod_{i: j\in i} \Lambda_i^* (n)^{m^\circ_{ji}}=\prod_{i\in\mI}
\Lambda_i^*(n)^{n^\circ_i}.
\end{equation*}
As the following proposition shows, this is the limit of the term (b).

\begin{proposition}\label{main prop}
For $n^\circ\in\bZ_+^I$, $n\in\bR_+^I$ and $\{ u_c\}_{c\in\bN}$ some bounded sequence such that
$c n + u_c\in\bZ_+^I$
\begin{equation*}
  \frac{{B}(c n +u_c -n^\circ)}{{B}(c n + u_c)}\xrightarrow[c\rightarrow\infty]{} \prod_{i\in \mI}
\Lambda^{*}_i(n)^{n^\circ_i}.
\end{equation*}
\end{proposition}
\begin{proof}
From Proposition \ref{conv prop}, we have that
\begin{equation}\label{thrm72}
 \frac{{B}(cn+u_c-e_i)}{{B}(cn+u_c)}\xrightarrow[c\rightarrow\infty]
{ }\Lambda^{*}_i(n),
\end{equation}
for any $n\in\bR_+^I$ and any bounded sequence $\{u_c\}_{c\in\bN}$ such
that $cn+u_c\in\bZ_+^I$. Let $K=\sum
_i n_i^\circ$, Let $e_{i(1)},...,e_{i(K)}$
be a finite sequence of unit vectors and $n(0),...,n({K})$ be a sequence of vectors in
$\bZ_+^I$ such that
\begin{align*}
n(0)&=0,\quad n(K)=n^\circ\\ 
n(k)&=n(k-1)+e_{i(k)},\quad k=1,...,K.
\end{align*}
Applying \eqref{thrm72}, we have that
\begin{equation*}
 \frac{{B}(c n + u_c-n^\circ)}{{B}(cn + u_c)}=\prod_{k=0}^{K-1}
\frac{{B}(c n + u_c-n(k)-e_{i(k+1)})}{{B}(c n + u_c-n(k))} \xrightarrow[c\rightarrow\infty]{}
\prod_{k=1}^{K}
\Lambda^{*}_{i(k)}(n)= \prod_{i\in \mI}
\Lambda^{*}_i(n)^{n^\circ_i},
\end{equation*}
as required.  
\end{proof}
As the above proposition holds for any closed queueing network and because Lemma \ref{2nd closed
lemma} holds, we can say that
\begin{equation}\label{B bar conv}
  \frac{\bar{B}(c n + u_c -n^\circ)}{\bar{B}(c n + u_c)}\xrightarrow[c\rightarrow\infty]{} \prod_{i\in \mI}
\Lambda^{*}_i(n)^{n^\circ_i}=\prod_{j\in\mJ^\circ} \prod_{i: j\in i} \Lambda_i^* (n)^{m^\circ_{ji}}.
\end{equation}
This gives the convergence of term (b) in expression \eqref{3terms}.

We now study term (c) in expression \eqref{3terms}. We can note that this term is exactly
the probability that all non-bottleneck queues are empty.
\begin{lemma}\label{lemma4}
\begin{equation*}
 \pi(\{m_j=0, \forall j\in \mJ^\circ \}|n)=\dfrac{\bar{B}(n)}{B(n)}.
\end{equation*}
\end{lemma}
\begin{proof}
We have the following equations:
\begin{align*}
  \pi(\{m_j=0, \forall j\in \mJ^\circ \}|n)&=\sum_{\substack{m\in\mS(n):\\ m_j=0, j\in\mJ^\circ}}
\dfrac{1}{B(n)}\prod_{j\in\mJ} \left(\binom{m_j}{m_{ji}:j\in i} \prod_{i: j\in
i}\mu_{ji}^{-m_{ji}}\right)\\
&=\sum_{\substack{(0,\bar{m})\in\mS(n)}}
\dfrac{1}{B(n)}\prod_{j\in\bar{\mJ}} \left(\binom{\bar{m}_j}{\bar{m}_{ji}:j\in i}
\prod_{i: j\in
i}\mu_{ji}^{-\bar{m}_{ji}}\right)\\
&= \sum_{\bar{m}\in\bar{\mS}(n)} \frac{1}{B(n)}\prod_{j\in\bar{\mJ}} 
\left(\binom{\bar{m}_j}{\bar{m}_{ji}:j\in i} \prod_{i: j\in
i}\mu_{ji}^{-\bar{m}_{ji}}\right)=\frac{\bar{B}(n)}{B(n)}.
\end{align*}
 
\end{proof}

Although it is difficult to directly deal with events of the form $\{m_j=0, \forall j\in \mJ^\circ
\}$, we can deal with events of the form $\{m_j>0, \forall j\in \mJ' \}$ where $\mJ'\subset
\mJ^\circ$.

\begin{lemma} \label{first prod lemma} For $\mJ'\subset \mJ^\circ$,
\begin{equation*}
 \pi( \{m_j>0, \forall j\in \mJ' \} | c n + u_c)
\xrightarrow[c\rightarrow \infty]{}
\prod_{j\in\mJ'} {U^{*}_j(n)}
\end{equation*}
where we define
\begin{equation}
\label{eq:U_jstar}
 U^{*}_j(n)\bydef\sum_{i:j\in i} \frac{\Lambda^*_i(n)}{\mu_{ji}}.
\end{equation}
\end{lemma}
\begin{proof}
To prove this lemma, we consider the explicit stationary distribution of a closed multi-class queueing network \eqref{eq:piexplicit}.

Recall $s=(s_j:j\in\mJ)$ where
$s_j=(i_j(k): k=1,...,m_j)\in\mI^{m_j}$ keeps
track of the exact position and route type of each customer within a queue.
With this state representation, we can calculate the probability that the customer at the head of each queue $j\in\mJ'$ is
from a specific route type $r(j)\in\mI$. In particular, if we let $r(j)$ be the route type of the
customer at the head of queue $j\in\mJ'$ and we let the vector $n'$ give the number of customers of
each route type at the heads of these queues then we can see that
\begin{align*}
 \pi(\{i_j(1)=r(j), j\in
\mJ'\}|n)&=\sum_{\substack{s\in\mX(n):\\i_j(1)=r(j), j\in
\mJ' }}\dfrac{1}{B(n)}\prod_{j\in\mJ}\prod_{i:j\in
i}\mu_{ji}^{-m_{ji}}\\
&=\bigg[\prod_{j\in\mJ'} \frac{1}{\mu_{jr(j)}} \bigg]\times
\sum_{s'\in\mX(n-n')}\dfrac{1}{B(n)}\prod_{j\in\mJ}\prod_{i:j\in
i}\mu_{ji}^{-m'_{ji}}\\
&=\bigg[\prod_{j\in\mJ'} \frac{1}{\mu_{jr(j)}}\bigg] \times \frac{{B}(n-n')}{{B}(n)}.
\end{align*}
In the second inequality above, we factor out the multiplicative terms corresponding to the heads of
each queue in $\mJ'$. We then notice the remaining states that must be summed over are all the
states where there are $n-n'$ customers on each route. The resulting sum then gives the normalizing
constant when there are $n-n'$ customers on each route.

Next, by Proposition \ref{main prop}, we have that
\begin{equation}\label{head queue lim}
  \pi(\{i_j(1)=r(j), j\in \mJ'\}|n) = \bigg[\prod_{j\in\mJ'} \frac{1}{\mu_{jr(j)}}\bigg] \times
\frac{{B}(n-n')}{{B}(n)} \xrightarrow[c\rightarrow\infty]{} \prod_{j\in\mJ'}
\frac{\Lambda_{r(j)}^*(n)}{\mu_{jr(j)}}.
\end{equation}
The events $\{i_j(1)=r(j), j\in \mJ'\}$ are disjoint for different choices of $r'=(r(j):j\in\mJ')$.
For a queue to be nonempty, there must be some customer at its head. Thus,
\begin{equation*}
 \bigcup_{r'\in \mI^{J'}}\{i_j(1)=r(j), j\in \mJ'\}= \{ m_j>0, j\in\mJ\},
\end{equation*}
and consequently using \eqref{head queue lim}, we have
\begin{align*}
 \pi(\{m_j>0, j\in\mJ\}|n) =& \sum_{r'\in \mI^{J'}}\pi(\{i_j(1)=r(j), j\in \mJ'\}|n)\\
=&\sum_{r'\in \mI^{J'}}\bigg[\prod_{j\in\mJ'} \frac{1}{\mu_{jr(j)}}\bigg] \times
\frac{{B}(n-n')}{{B}(n)}\\
\xrightarrow[c\rightarrow \infty]{} & \sum_{r'\in \mI^{J'}} \prod_{j\in\mJ'}
\left(\frac{\Lambda^*_{r(j)}(n)}{\mu_{jr(j)}}\right)
= \prod_{j\in\mJ'}
\left(\sum_{r\in \mI} \frac{\Lambda^*_{r}(n)}{\mu_{jr}}\right)=\prod_{j\in\mJ'}
U^*_{j}(n)
\end{align*}
as required.
 
\end{proof}
Now we are in the position to prove the convergence of term (c) in expression \eqref{3terms}.

\begin{proposition}\label{normalize const lemma}
\begin{equation*}
\frac{\bar{B}(c n  + u_c)}{B(c n+ u_c)} \xrightarrow[c\rightarrow \infty]{}
\prod_{j\in\mJ^\circ}
\left(1-{U^{*}_j(n)}\right)
\end{equation*}
\end{proposition}
\begin{proof}
In the following expression, we use  Lemma \ref{lemma4}; we apply the Inclusion-Exclusion
Principle
\eqref{inc-ex}; we apply Lemma \ref{first prod lemma} \eqref{use prod
lemma}; and then we notice the resulting summation is  $\prod_{j\in\mJ^\circ}
\left(1-{U^{*}_j(n)}\right)$ expanded:
\begin{align}
\dfrac{\bar{B}(c n + u_c)}{B(c n + u_c)}=&  \pi(\{m_j=0, \forall j\in \mJ^\circ \}|c n + u_c) \\
=&1-\pi\bigg( \bigcup_{j\in\mJ^\circ} \{ m_j > 0 \} \Big| c n + u_c\bigg) \notag \\
= &\sum_{k=0}^{J^\circ} \sum_{\substack{j_1,...,j_k\in\mJ^\circ:\\j_1<...<j_k}} (-1)^k
\pi(\{m_j > 0, j=j_1,...,j_k\}|c n + u_c) \label{inc-ex}\\
\xrightarrow[c\rightarrow\infty]{} &\sum_{k=0}^{J^\circ}
\sum_{\substack{j_1,...,j_k\in\mJ^\circ :\\j_1<...<j_k}} (-1)^k\prod_{j=j_1,...,j_k}
{U^{*}_j(n)} = \prod_{j\in\mJ^\circ}\left(
1-{U^{*}_j(n)} \right). \label{use prod lemma}
\end{align}
In expression \eqref{inc-ex}, the $k=0$ summand is understood to equal $1$.
  \end{proof}

We have now discovered the limiting behavior of terms (a), (b), and (c), and we can prove
Theorem~\ref{Thrm1}.

\begin{proof}[Proof of Theorem \ref{Thrm1}]
We apply to the equality \eqref{3terms} the results \eqref{B bar conv} (which we
derived from Proposition \ref{normalize const lemma}) and  Proposition \ref{main prop}:
 \begin{align*}
 \pi^\circ (m^\circ|c n + u_c) =& {\dfrac{\bar{B}(c n + u_c)}{B(c n + u_c)}}\times
{\frac{\bar{B}(c n + u_c-n^\circ)}{\bar{B}(c n + u_c)}} \times
{\prod_{j\in\mJ^\circ} \left(\binom{m^\circ_j}{m^\circ_{ji}:j\in i} \prod_{i: j\in
i} \frac{1}{\mu_{ji}^{m^\circ_{ji}}}\right)}\\
\xrightarrow[c\rightarrow\infty]{}&   
\prod_{j\in\mJ^\circ}\left( 1-{U^{*}_j(n)} \right) \times
\prod_{j\in\mJ^\circ} \prod_{i: j\in i} \Lambda_i^*
(n)^{m^\circ_{ji}} \times
{\prod_{j\in\mJ^\circ} \left(\binom{m^\circ_j}{m^\circ_{ji}:j\in i} \prod_{i: j\in
i} \left(\frac{1}{\mu_{ji}}\right)^{m^\circ_{ji}}\right)}\\
=&\prod_{j\in\mJ^\circ}\left(
1-{U^{*}_j(n)} \right) \times
{\prod_{j\in\mJ^\circ} \left(\binom{m^\circ_j}{m^\circ_{ji}:j\in i} \prod_{i: j\in
i} \left(\frac{\Lambda_i(n)}{\mu_{ji}}\right)^{m^\circ_{ji}}\right)}\\
=& \pi^\circ_{\Lambda^*(n)}(m^\circ).
\end{align*}
This proves that our network's non-bottleneck queues become independent.
  \end{proof}

We have now shown, for our stationary closed queueing network, convergence in distribution of non-bottleneck queues to a stationary open queueing network.  We now wish to show convergence of queue sizes in expectation, Theorem~\ref{Thrm1_expectation}.

Building on the above results, we now develop a proof for Theorem~\ref{Thrm1_expectation}.
Once again, we gain the result by considering a modified version of our original queueing network.
 Recalling the normalizing constant $B(n)$ \eqref{Bn}, let $B^{+j}(n)$ be the normalizing constant of the closed queueing network that is obtained by adding a
replica of queue $j$. This replica is added immediately after queue $j$. Customers leaving $j$ will immediately join this replica queue, and then, after service at the replica queue, customers will continue on their route.
We can now express mean queue lengths in terms of these normalizing constants.
\begin{lemma}
\label{lm:exp_B}
$$E M_{ji}^\circ(n) = \frac{1}{\mu_{ji}} \frac{B^{+j}(n-e_i)}{B(n)}.$$
\end{lemma}
\begin{proof}
For all $j'\in\mJ,i'\in \mI$, we have
\begin{eqnarray}
E M_{j'i'}^\circ(n) 
&& \bydef \frac{1}{B(n)}\sum_{m\in \mS(n):m_{j'i'}>0} m_{j'i'} \prod_{j\in\mJ} m_j! \prod_{i:j\in i } \frac{\mu_{ji}^{-m_{ji}}}{m_{ji}!} \\ 
%
%
&& = \frac{\mu_{j'i'}^{-1}}{B(n)}\sum_{m\in \mS(n):m_{j'i'}>0} \prod_{{\scriptstyle j\in\mJ:}\atop{\scriptstyle j\neq j'}} m_j! \prod_{i:j\in i } \frac{\mu_{ji}^{-m_{ji}}}{m_{ji}!} \times m_{j'}! \prod_{i:j'\in i,i\neq i' } \frac{\mu_{j'i}^{-m_{j'i}}}{m_{j'i}!} \times \frac{\mu_{j'i'}^{-m_{j'i'}+1}}{(m_{j'i'}-1)!}\\ 
&& = \frac{\mu_{j'i'}^{-1}}{B(n)}\sum_{m\in \mS(n-e_{i'})} \prod_{{\scriptstyle j\in\mJ:}\atop{\scriptstyle j\neq j'}} m_j! \prod_{i:j\in i } \frac{\mu_{ji}^{-m_{ji}}}{m_{ji}!} \times (m_{j'}+1) ! \prod_{i:j'\in i} \frac{\mu_{j'i}^{-m_{j'i}}}{m_{j'i}!} \\ 
&& =   \frac{\mu_{j'i'}^{-1}}{B(n)}\sum_{m\in \mS(n-e_{i'})} \prod_{{\scriptstyle j\in\mJ:}\atop{\scriptstyle j\neq j'}} m_j! \prod_{i:j\in i } \frac{\mu_{ji}^{-m_{ji}}}{m_{ji}!}
\times \sum_{k=0}^{m_{j'}}
\binom{m_{j'}}{m_{j'i}:i\ni j} \prod_{i:j'\in i } \mu_{j'i}^{- m_{j'i}} \label{eq:using.vander}
\ .
\end{eqnarray}
Using the multinomial Vandermonde convolution we can rewrite the last sum in \eqref{eq:using.vander}
in the following way
\begin{eqnarray}
&&\sum_{k=0}^{m_{j'}}
\binom{m_{j'}}{m_{j'i}:i\ni j} \prod_{i:j'\in i } \mu_{j'i}^{- m_{j'i}}\\
 =  &&
\sum_{k=0}^{m_{j'}}
\sum_{{\scriptstyle \tilde m_1,\tilde m_2\in\mathbb{Z}_+^{I}:\sum_{i}\tilde m_{1i}=k,}\atop{\scriptstyle \tilde m_{1i}+\tilde m_{2i}=m_{j'i},\forall i}}
\binom{k}{\tilde m_{1i}:i\ni j'}\binom{m_{j'}-k}{m_{j'i}-\tilde m_{1i}:i\ni j'} \prod_{i:j'\in i } \mu_{j'i}^{-m_{j'i}} \label{eq:vander}\\
%
%
 =&&
\sum_{{\scriptstyle \tilde m_1,\tilde m_2\in\mathbb{Z}_+^{I}:}\atop{\scriptstyle \tilde m_{1i}+\tilde m_{2i}=m_{j'i},\forall i}}
\binom{\sum_i \tilde m_{1i}}{\tilde m_{1i}:i\ni j'} \prod_{i:j'\in i } \mu_{j'i}^{-\tilde m_{1i}}
\binom{\sum_i \tilde m_{2i}}{\tilde m_{2i}:i\ni j'} \prod_{i:j'\in i } \mu_{j'i}^{-\tilde m_{2i}} 
\ .
\end{eqnarray}
We can interpret the above identity as adding an extra queue to our network
and hence it follows that 
\begin{equation}
E M_{j'i'}^\circ(n) 
 = \frac{\mu_{j'i'}^{-1}}{B(n)}\sum_{m\in \mS^{+j'}(n-e_{i'})} \prod_{j\in\mJ^{+j'}} m_j! \prod_{i:j\in i } \frac{\mu_{ji}^{-m_{ji}}}{m_{ji}!} \label{eq:ss_replica}
 = \frac{\mu_{j'i'}^{-1} }{B(n)}B^{+j'}(n-e_{i'})
\ .
\end{equation}
In the above expression, $\mS^{+j'}(n-e_{i'})$ and $\mJ^{+j'}$ are, respectively, the state space and set of queues obtained when adding a replica of queue $j'$.
  \end{proof}

With Lemma~\ref{lm:exp_B}, we can prove Theorem \ref{Thrm1_expectation}.

\begin{proof}[Proof of Theorem \ref{Thrm1_expectation}]
%
%
By assumption, $j\in\mJ^\circ$ -- it is a non-bottleneck queue. Notice that the replica queue added for $B^{+j}$ corresponds to a repeated constraint in optimization \eqref{PF Opt}, which by Theorem~\ref{Thrm1} implies that the replica queue must also be a non-bottleneck queue.

By Lemma~\ref{lm:exp_B} and the definition of $\bar{B}(n)$ given in \eqref{eq:barB}, the
following equation holds
\begin{equation} \label{eq:BplusB}
\lim_{c\to\infty} E M_{ji}^\circ(c n + u_c) 
=
\frac{1}{\mu_{ji}}\lim_{c\to\infty} \frac{\bar B(c n + u_c)}{B(c n + u_c)} \frac{\bar B(cn-e_i+ u_c)}{\bar B(c n + u_c)} \frac{B^{+j}(cn-e_i+ u_c)}{\bar B(cn-e_i+u_c)}
\end{equation}
Now, using Lemma~\ref{2nd closed lemma},
\begin{equation}
\label{eq:Bpluslimit1}
\lim_{c\to\infty} \frac{\bar B(cn-e_i+u_c)}{\bar B(c n + u_c)} = \Lambda_i^*(n) 
\end{equation}
Since the replica queue $j$ is not a bottleneck, we can apply Proposition~\ref{normalize const lemma} twice to find that 
\begin{equation}
\label{eq:Bpluslimit2}
\lim_{c\to\infty} \frac{\bar B(c n + u_c)}{B(c n + u_c)} \frac{B^{+j}(cn-e_i+u_c)}{\bar B(cn-e_i+u_c)} =
\frac{1}{1-{U^{*}_j(n)}},
\end{equation}
where we recall that $U^{*}_j(n)$ is defined in \eqref{eq:U_jstar}.
Substituting \eqref{eq:Bpluslimit1}-\eqref{eq:Bpluslimit2} in \eqref{eq:BplusB}, we get
\begin{equation}\label{lm:lims}
\lim_{c\to\infty} E M_{ji}^\circ(c n + u_c) =  \dfrac{\Lambda_i^*(n)/\mu_{ji}}{1-U_j^*(n)}.
\end{equation}
For a non-bottleneck queue, $j\in \mJ^\circ$, $U^*_j(n)= \sum_{i:j\in i} \Lambda^*_i(n)/{\mu_{ji}}<1$, with some standard calculations on  one can verify that
\begin{equation}
\label{eq:QL_open}
E M_{\Lambda^*(n),ji}^\circ = \dfrac{\Lambda^*_i(n)/\mu_{ji}}{1-U^*_j(n)},
\end{equation}
for all $i\in\mI$ and $j\in \mJ^\circ$. Thus together \eqref{eq:QL_open} and \eqref{lm:lims} imply the required result
\begin{equation*}
\lim_{c\to\infty} E M_{ji}^\circ(c n + u_c) = E M_{\Lambda^{*}(n),ji}^\circ
\end{equation*}
for all $i\in\mI$ and $j\in \mJ^\circ$.
  \end{proof}

\subsection{Proof of fluid limit.}\label{proof2}
In this section, we prove Theorem \ref{main fluid theorem}. We first introduce and recall some notation before proceeding with a proof.

Let $\mI(j)\bydef\{i\in\mI: j \in i\}$ and $I(j)\bydef|\mI(j)|$. Assuming that set
$\mI(j)$ is ordered, we also denote by $i_j(l)$, with $l=1,\ldots,I(j)$, the function that enumerates its elements.
We define for each $j \in \mJ$, an independent Poisson marked point process $N_j$ with intensity 
$\mu_j \, dt \otimes du$ on $\bR \times [0,1]$, where $\mu_j = \max \{\, \mu_{ji}: i \in \mI(j)\}$, and the function $\chi_j(m_j,u)$ on $\bN^{I(j)}\times(0,1)$ in the following way:
$$
\chi_j(m_j,u)
=\left\{\begin{array}{ll}
0 	& \mbox{ if $m_j = 0$ }  \\
\\
i_j(l)	& \mbox{if } 
\sum_{h=0}^{l-1} m_{jh} \, \mu_{jh}
    \leq m_j \, \mu_j \times u  \leq
	\sum_{h=0}^l m_{jh} \, \mu_{jh} \\
\\
0 	& \mbox{if } 
\sum_{h=0}^{I(j)} m_{jh} \, \mu_{jh}  
	\leq m_j \, \mu_j \times u\leq m_j \, \mu_j.
\end{array}\right.
$$
Therefore, the network process can be written as
$$
dM_{ji}(t) 
=  1\{ \chi_{j'}(M_{j'}(t-),U_{N_{j'}(t)}) = i \} \, dN_{j'}(t)  
-  1\{ \chi_j(M_j(t-),U_{N_{j}(t)}) = i \} \, dN_{j}(t) 
$$
where $j'$ denotes the queue just before $j$ on route $i$, and
the uniform independent random variables $U_k$ are generated by the second coordinate of
the marked point process. In the following we will use also the
notations $\hat j$ and $\hat j'$ to denote the first queue on route $i$ before the
queues $j$ and $j'$ respectively that are non empty. In integral form, we have
$$M_{ji}(t) = M_{ji}(0) + \mL_{j'i}(t) - \mL_{ji}(t) + \aL_{j'i}(t) - \aL_{ji}(t) $$
where $\mL_{ji}(t)$ is the martingale 
$$
\mL_{ji}(t)
 = \int_0^t \int_0^1 1\{ \chi_j(M_j(s-),u) = i \} [dN_j(s,u) - \mu_j \, du \, ds ],
$$
and the process $\aL_{ij}(t)$ is given by
$$
\aL_{ji}(t)  = \int_0^t \int_0^1 \mu_j \, 1\{\chi_j(M_j(s-),u) = i \} \, du \, ds.
$$
From the equation above, noticing that 
$
\int_0^1 1\{ \chi_j(M_j(s,\, u) = i \} \, du 
= (\mu_{ji} \, M_{ji}(s))/(\mu_j \, M_j(s))$ and with the convention that $0/0=0$,
we get that 
$$
\frac{1}{\mu_{ji}}\aL_{ji}(t)
= \int_0^t \frac{\mu_j}{\mu_{ji}} \int_0^1  \, 1\{\chi_j(M_j(s-),u) = i \} \, du \, ds \\
= \int_0^t \frac{M_{ji}(s-)}{M_j(s-)} \, ds. 
$$
Summing over $i:j\in i$ we have that for $t'\leq t''$
\begin{equation}
\sum_{i:j\in i} \frac{1}{\mu_{ji}} \left(\aL_{ji}(t'')-\aL_{ji}(t')\right)
=
\int_{t'}^{t''} \sum_{i:j\in i} \frac{M_{ji}(s-)}{M_j(s-)} \, ds 
=
\int_{t'}^{t''} \sum_{i:j\in i} 1\{M_j(s-)\not=0\} \, ds \leq t''-t'  \label{Lambda.Lipschitz},
\end{equation}
which gives the Lipschitz condition for the process $\aL_{ji}(t)$.

We define the scaled process $\bar M(n \, c, \, t) = c^{-1} \, M(n \, c, \, c\, t)$, such that,
having $M_{ji}(0) = c \, n_{ji}$, it has the $ji$-component given by
$$
\bar M_{ji}(n \, c, \, t)
= n_{ji}
+ \frac{ \mL_{j'i}(c t) - \mL_{ji}(c t)}{c} 
+ \frac{ \aL_{j'i}(c t) - \aL_{ji}(c t)}{c}.
$$
Using the same steps as before, we can rewrite the process $c^{-1} \aL_{ji}(ct)$ in the following way
\begin{eqnarray*}
\frac{\aL_{ji}(ct)}{c} 
&=& \frac{1}{c} \int_0^{c\,t}
 \int_0^1 \mu_j \, 1\{ \chi_j(M_j(s), u) = i \} \,  du \, ds \\
&=& \int_0^t
 \int_0^1 \mu_j \, 1\{ \chi_j(M_j(c\, s),\, u) = i \}  \,  du \, ds 
= \int_0^t  \frac{M_{ji}(c\,s)}{M_j(c\,s)} \, \mu_{ji} \,  ds  \label{eq:aji} \ .
\end{eqnarray*}

\begin{proposition}
Given $n \in \mathbb{N}^{J\times I}$,  the processes $\{(\bar M (n \, c, \, t), t>0)\}_{c>0}$, with $M (n \, c, \, 0) = n \, c$,  are
tight.
\end{proposition}
 \begin{proof}
By the triangular inequalities for metrics, to prove tightness of the multidimensional process 
$\{(\bar M (n \, c, \, t), t>0)\}_{c>0}$, it is enough to show that any of its coordinate processes
are tight. By Theorem~C.9 in Robert \cite{robert:2000}, this can be done by showing that: given $i,j$, for
any fixed $T, \eta>0$, 
there exists $\delta>0$ such that
\begin{equation}\label{conv}
\Pr\{\omega_{\bar M_{ji}(n \, c, \, \cdot)}(\delta) >\eta \}<\epsilon \,
\end{equation}
where, for a given function $f$, the modulus of continuity on $[0,T]$ is defined as
$$\omega_f(\delta) = \sup\{ |f(t)-f(s)| : s,t\leq T, |t-s|<\delta\}.$$
Since 
$\bar M_{ji}(n \, c, \, t) = n + c^{-1} \, \Delta\mL_{ji}(c\, t)
+ c^{-1} \, \Delta\aL_{ji}(c\, t)-(\mL_{ji}(c\, t))$ with 
$\Delta\mL_{ji}(t)=\mL_{j'i}(t)-\mL_{ji}(t)$ and 
$\Delta\aL_{ji}(t)=\aL_{j'i}(t)-\aL_{ji}(t)$,
it is enough to prove that relation (\ref{conv}) is valid separately for the processes
$c^{-1} \Delta\mL_{ji}(c\,t)$ and $c^{-1} \Delta\aL_{ji}(c\,t)$.
We have, for $T>0$ and $\eta >0$,
\begin{eqnarray}\label{Lambda.tightness}
&&
\Pr\Big\{\sup_{s,t\leq T; |t-s|<\delta} 
\Big|\frac{\Delta\aL_{ij}(c \, t)}{c}-\frac{\Delta\aL_{ij}(c \, s)}{c}\Big| >\eta \Big\} 
\nonumber \\
&&\quad\quad =
\Pr\Big\{\sup_{s,t\leq T; |t-s|<\delta} \Big|
\int_s^t \left(
  \frac{M_{j'i}(c\,u)}{M_{j'}(c\,u)} \, \mu_{j'i} 
-  \frac{M_{ji}(c\,u)}{M_j(c\,u)} \, \mu_{ji}
 \right) \,  du 
\Big| >\eta \Big\} \nonumber \\
&&\quad\quad \leq
\Pr\Big\{\sup_{s,t\leq T; |t-s|<\delta} 
\int_s^t \left|
  \frac{M_{j'i}(c\,u)}{M_{j'}(c\,u)} \, \mu_{j'i} 
-  \frac{M_{ji}(c\,u)}{M_j(c\,u)} \, \mu_{ji} 
 \right| \,  du  >\eta \Big\} \nonumber \\
&&\quad\quad \leq
\Pr\Big\{\sup_{s,t\leq T; |t-s|<\delta} 
 2 |t-s| \mu >\eta \Big\} = 0 
\end{eqnarray}
with $\delta < \eta/(2\mu)$ and where $\mu = \max\{i,j:\mu_{ji}\}$.

For the martingale  $c^{-1} \, \Delta\mL_{ji}(c\, t)$, we have that
\begin{equation}
\Pr\Big\{\sup_{s,t\leq T; |t-s|<\delta} 
\Big|\frac{\Delta\mL_{ji}(c\, t)}{c}-\frac{\Delta\mL_{ji}(c\, s)}{c}\Big|
>\eta \Big\}
=
\Pr\Big\{\sup_{t\leq T} \frac{\Delta\mL_{ji}(c\, t)}{c} > \frac{\eta}{2} \Big\} \nonumber
\leq
\frac{4}{c^2\eta^2} \bE\left[(\Delta\mL_{ji}(c\, t))^2\right]
\end{equation}
where the last inequality follows by applying the Doob's inequality.
Having that $\bE\left[(\Delta\mL_{ji}(c\, T))^2\right] \leq 2 \mu c T$ we have that
for  $c >  8 \mu T/(\epsilon \eta^2)$ the probability is bounded above by $\epsilon$, as required.%
 \end{proof}

The tightness property ensures the relative compactness, therefore from every sequence
$\{(\bar M (n \, c, \, t), t>0)\}_{c>0}$, with $M (n \, c, \, 0) = n \, c$ and $n$ fixed,
it is possible to extract a convergent subsequence. 
The following proposition ensures that any limit process will be given by a fluid solution.

\begin{proposition}
Assume that a sequence of processes $\{(\bar M (n \, c, \, t), t>0)\}_{c>0}$ converges to a limit  
process $m(t)$ as $c\to\infty$. Then, $m(t)$ is almost surely continuous and it is a fluid solution
as defined in Definition \ref{def:fluid.sol}.
\end{proposition}
 \begin{proof}
Using the Skorohod's Representation theorem, see Robert \cite[Theorem C.8]{robert:2000},
we can assume that all the elements of the sequence
are random processes defined on the same probability space with probability $\bP$ 
and the convergence is $\bP$-a.s.

Since we have that for any $c$
$$
\bar M_{ji}(n \, c, \, t) = n_{ji} + \frac{\Delta\mL_{ji}(c\, t)}{c}  
+  \frac{\aL_{j'i}(c\, t)}{c} -  \frac{\aL_{ji}(c\, t)}{c},
$$
passing to the limit and using bounded convergence for the integral we get
$$
\bar m_{ji}(t) = n_{ji} +  \aL_{j'i}(t) -  \aL_{ji}(t) 
$$
where 
$$
\aL_{ji}(t) = \lim_{c\to\infty}   \int_0^t  \frac{M_{ji}(c\,u)}{M_j(c\,u)} \, \mu_{ji}  \,  du,
$$
which exists by (\ref{Lambda.tightness}). In particular, if $m_j(t)>0$ for some $t$, it will be 
positive in a neighborhood $B(t)$ of $t$ by continuity.
It follows that for $t'<t''$ and $t',t''\in B(t)$,
$$
\aL_{ji}(t'') - \aL_{ji}(t') = \lim_{c\to\infty}   \int_{t'}^{t''} 
\frac{M_{ji}(c\,u)}{M_j(c\,u)} \, \mu_{ji}  \,  du 
=    \int_{t'}^{t''} \lim_{c\to\infty}
\frac{M_{ji}(c\,u)}{M_j(c\,u)} \, \mu_{ji}  \,  du 
=  \int_{t'}^{t''} \frac{m_{ji}(u)}{m_j(u)} \, \mu_{ji}  \,  du,
$$
where in the second equality we have used the bounded convergence theorem, which implies
equation \eqref{fluidequns4}, i.e. 
$$\dL_{ji}(t) = \frac{d\aL_{ji}(t)}{dt} = \mu_{ji}  
\, \frac{m_{ji}(t)}{m_j(t)} \quad \mbox{ as } m_j (t)>0.$$
The additional conditions satisfied by $\aL_{ji}(t)$ easily follow as equivalent property holds for
the approximating processes $\{c^{-1} \, \aL_{ji}(c\,t), c>0\}$, in particular the Lipschitz
condition follows by \eqref{Lambda.Lipschitz}.
   \end{proof}

\subsection{Proof of convergence of bottleneck queues.}\label{proof3}

We now focus on proving Theorem~\ref{Thrm3}. For technical reasons that we will explain shortly, our proof is quite involved and requires a number of lemmas.
For a solution $(m(t), \aL(t))$ of the fluid model \eqref{fluidequns}, let $\dL_{ji}(t)$ be the derivative of
$\aL_{ji}(t)$, when it exists. Let also $l_{ji}$ be the queue before queue $j$ on route $i$ and
$x_{ji}$ be the next queue after queue $j$ on route $i$.

We recall the function $\beta(m(t))$:
\begin{equation}\label{Lyabeta}
\beta(m(t)) \bydef \sum_{j\in\mJ} \sum_{\substack{i: j\in i\\ m_{ji}(t)>0}} m_{ji}(t) \log \frac{m_{ji}(t)\mu_{ji}}{m_{j}(t)}.
\end{equation}
The terms $m_{ji}(t)$ are Lipschitz and thus are differentiable for almost every $t$. So, if $m_{ji}(t)>0$ then we can differentiate the $ij^{th}$ summand of \eqref{Lyabeta}. However, if 
$m_{ji}(t)=0$, taking a derivative become a significantly more technical issue.
The following proposition ensures we can differentiate the summands of $\beta(m(t))$ and the subsequent lemma ensures summands have zero derivative when $m_{ji}(t)=0$.

\begin{proposition} \label{beta abs cont} For each $i$ and $j\in i$ and
for any time interval $[t_0,t]$ with $t>t_0>0$, there exists a constant $D>0$ such that for any
$t_1$, $t_2\in [t_0,t]$
\begin{equation}\label{lipt.cond}
|\beta(m(t_2)) - \beta(m(t_1))| \leq D |t_2 - t_1|\ . 
\end{equation}
Moreover, the same locally-Lipschitz condition applies to each summand of $\beta(m(t))$, \eqref{Lyabeta}.
\end{proposition}

\begin{lemma} \label{diff at zero}
For functions $f:\bR_+\rightarrow\bR_+$ and $p:\bR_+\rightarrow [0,1]$, if $t>0$ is such that
the derivative of $f$ exists at $t$, the derivative of $f\log p$ exists at $t$ and
$f(t)=0$ then 
\begin{equation*}
 \frac{df}{dt}=0 \qquad\text{and}\qquad \frac{df\log p}{dt} = 0.
\end{equation*}
\end{lemma}
Proposition \ref{beta abs cont} is based on Proposition 4.2 of Bramson \cite{Br96}. Both
Proposition \ref{beta abs cont} and Lemma \ref{diff at zero} are proven in Appendix \ref{appendix2}.
It will also be useful to have the following lemma, which is proven in Appendix \ref{appendix2}.
\begin{lemma} \label{lamba diff}
 For almost every $t$, $0< \dL_{ji}(t) \leq \mu_{max}$, where $\mu_{max}=\max\{ \mu_{ji}:
i\in\mI, j\in i\}$.
\end{lemma}

Recalling that $x_{ji}$ is the next queue on route $i$ after $j$ and that $l_{ji}$ is the queue
before queue $j$ on route $i$, we can now prove the following proposition.
\begin{proposition} \label{diff prop}
For almost every $t$,
\begin{equation}\label{beta derivative}
\frac{d\beta(m(t))}{dt}= -\sum_{i\in\mI}\sum_{j\in i} \dL_{ji}(t)\log
\frac{\dL_{ji}(t)}{\dL_{x_{ji}i}(t)}. 
\end{equation}
\end{proposition}
 \begin{proof}
Our processes in Proposition \ref{beta abs cont} are absolutely continuous and thus almost
everywhere
differentiable. So for almost every $t$, we can differentiate the terms
$\aL_{ji}(t)$, $m_{ij}(t)$, $m_j(t)$, $m_{ij}(t)\log ( m_{ji}(t)\mu_{ji}/m_j(t) )$, and
$\beta(m(t))$. Differentiating, we obtain
\begin{subequations}\label{beta diff}
\begin{align}
 \frac{d\beta(m(t))}{dt}
&=\sum_{j\in\mJ}  \sum_{\substack{i : j\in i\\ m_{ji}(t)>0}} \Big(\frac{dm_{ji}\log
m_{ji}}{dt}  -
\frac{dm_{ji}}{dt} \log \mu_{ji}
\Big) - \sum_{\substack{j\in\mJ:\\ m_{j}(t)>0}} \frac{dm_j\log m_j}{dt}\label{beta diff2}\\
&=\sum_{j\in\mJ}  \sum_{\substack{i : j\in i\\ m_{ji}(t)>0}} \Big(\frac{dm_{ji}}{dt} \log
m_{ji} + \frac{dm_{ji}}{dt} -
\frac{dm_{ji}}{dt} \log \mu_{ji}
\Big) - \sum_{\substack{j\in\mJ:\\ m_{j}(t)>0}} \left( \frac{dm_j}{dt}\log m_j + \frac{dm_j}{dt}
\right)\label{beta diff3} \\
&= \sum_{j\in\mJ}  \sum_{\substack{i : j\in i\\ m_{ji}(t)>0}} \frac{dm_{ji}}{dt} \log
\frac{ m_{ji}(t) \mu_{ji}}{m_j(t)}\label{beta diff4}\\
&=\sum_{i\in\mI}\sum_{j\in i} \Big( \dL_{l_{ji}i}(t)-\dL_{ji}(t) \Big)\log
\dL_{ji}(t)\label{beta diff5}\\
&=\sum_{i\in\mI}\sum_{j\in i} \dL_{ji}(t) \Big(\log
\dL_{x_{ji}i}(t)-\log \dL_{ji}(t)\Big)\label{beta diff6}\\
&=-\sum_{i\in\mI}\sum_{j\in i} \dL_{ji}(t)  \log \frac{\dL_{ji}(t)}{\dL_{x_{ji}i}(t)}.
\end{align}
\end{subequations}
For the above sequence of equalities, in equality \eqref{beta diff2}, we restrict our attention to summands with $m_{ji}(t)>0$ by applying Lemma \ref{diff at zero} and Proposition \ref{beta abs cont}.
Equality \eqref{beta diff4} holds by observing that $\sum_{i:j\in i}m_{ji}=m_j$ and canceling
terms. For equality \eqref{beta diff5}, we know by our fluid model assumption \eqref{fluidequns1}
that $m'_{ji}(t)=\dL_{l_{ji}i}(t)-\dL_{ji}(t)$. In addition, we note that if
$m_{ji}>0$ then, using \eqref{fluidequns4}, $\log
\frac{m_{ji}\mu_{ji}}{m_j}=\log \dL_{ji}(t)$ and if $m_{ji}(t)=0$ then $0=m'_{ji}(t)=
\dL_{l_{ji}i}(t)-\dL_{ji}(t)$ and $\dL_{ji}(t)>0$ (by Lemma \ref{lamba diff}). Thus, we may reintroduce the
$m_{ji}(t)=0$ terms in our summation. In equality \eqref{beta diff5}, for each route, we re-interpolate
the first term in our summation. 
 \end{proof}

The next lemma, found in Cover and Thomas \cite{CoTh91}, will be of key importance in bounding our
Lyapunov function.

\begin{lemma}[Pinsker's Inequality] For the relative entropy between two discrete probability 
distributions $p=(p_j)_j$ and $q=(q_j)_j$ with the same support:
 \begin{equation*}
D(p||q)=\sum_{j} p_j \log \frac{p_j}{q_j},
\end{equation*}
the following inequality holds
\begin{equation*}
 \sqrt{D(p||q)} \geq \sum_j | p_j - q_j|.
\end{equation*}
\end{lemma}
Applying Pinsker's inequality to Proposition \ref{diff prop} gives

\begin{lemma}\label{formal diff bound}
 For almost every $t$,
\begin{equation}\label{eq:formal diff bound}
 \frac{d\beta(m(t))}{dt} \leq -  \sum_{i\in \mI} \frac{1}{\mu_{max}|\mJ|}\sum_{j\in i}
\Big(\dL_{ji}(t) - \dL_{x_{ji}i}(t)\Big)^2, 
\end{equation}
where $|\mJ|$ is the size of set $\mJ$ and we recall that $\mu_{max}=\max\{ \mu_{ji}: i\in\mI, j\in
i\}$.
\end{lemma}
 \begin{proof}
Let $t$ be a time for which Proposition \ref{diff prop} holds and let  $\dL^\Sigma_i(t) =
\sum_{j\in i} \dL_{ji}(t)$, for $i\in\mI$. 
For each $i$, let $p_j=\dL_{ji}(t)/\dL^\Sigma_i(t)$ and
$q_j=\dL_{x_{ji}i}(t)/\dL^\Sigma_i(t)$. By Lemma \ref{lamba diff}, $p$ and $q$ both have
the same support. For each $i$, we applying Pinsker's Lemma
\begin{equation*}
 \sum_{j\in i} \frac{\dL_{ji}(t)}{\dL^\Sigma_i(t)} \log
\frac{\dL_{ji}(t)}{\dL_{x_{ji}i}(t)} = \sum_{j\in i} p_j \log \frac{p_j}{q_j} 
\geq \bigg( \sum_{j\in i} |p_j-q_j| \bigg)^2 \geq \sum_{j\in i} |p_j-q_j|^2 =
\frac{1}{(\aL_i^\Sigma(t))^2} \sum_{j\in i} \Big( \dL_{ji}(t) -
\dL_{x_{ji}i}(t)\Big)^2
\end{equation*}
Multiplying the left and right of this inequality by $-\aL_i^\Sigma(t)$, summing over
$i\in\mI$ gives
\begin{equation*}
 \frac{d\beta(m(t))}{dt}=-\sum_{i\in\mI}\sum_{j\in i} \dL_{ji}(t)  \log
\frac{\dL_{ji}(t)}{\dL_{x_{ji}i}(t)}  \leq -\sum_{i\in\mI}\frac{1}{\aL_i^\Sigma(t)}
\sum_{j\in i} \Big( \dL_{ji}(t) -
\dL_{x_{ji}i}(t)\Big)^2.
\end{equation*}
Recall that from Lemma \ref{lamba diff} that $\dL_{ji}(t)\leq \mu_{max}$ thus
$\dL^\Sigma_{i}(t)\leq |\mJ|\mu_{max}$. Applying this bound to $\dL^\Sigma_{i}(t)$ the above
equation gives the required result \eqref{eq:formal diff bound}.
 \end{proof}

We define $m^*$ to be a solution to the optimization problem
\begin{equation}\label{optimize beta}
 \text{minimize}\quad \beta(m) \quad \text{subject to}\quad \sum_{j\in i} m_{ij}=n_i,\quad i\in\mI
\quad \text{over}\quad m_{ji}\geq 0,\quad i\in \mI,\; j\in i.
\end{equation}
As we discussed, we expect the path of the $m(t)$ to converge to the optimal value of the
optimization. To conduct further analysis, we characterize the dual of this problem.

\begin{lemma}\label{duality-lemma}
The dual of the optimization \eqref{optimize beta} is
\begin{equation*}
\text{maximize}\quad \sum_{i\in\mI} n_i\log \dL_i \quad \text{subject to} \quad \sum_{i: j\in
i} \frac{\dL_i}{\mu_{ji}} \leq 1 \quad \text{over}\quad \dL_i \geq 0,\;\; i\in\mI. 
\end{equation*} 
\end{lemma}
 \begin{proof}
Taking Lagrange multipliers $\lambda\in\bR^I$, its Lagrangian function is,
\begin{align*}
L(m,\lambda)&=\sum_{j\in\mJ: m_j>0} \sum_{i\in\mI} m_{ji}\log \frac{m_{ji}\mu_{ji}}{m_j} +
\sum_{i\in\mI} \lambda_i\left(n_i-\sum_{j:j\in i} m_{ji}\right)\\
&=\sum_{j\in\mJ: m_j>0} \sum_{i\in\mI} m_{ji}\log \frac{m_{ji}\mu_{ji}}{m_je^{\lambda_i}} +
\sum_{i\in\mI} \lambda_in_i\\
&=\sum_{j\in\mJ: m_j>0} m_j D(p^{(j)}||q^{(j)}) - \sum_{j\in\mJ: m_j>0}m_j \log \Big( \sum_{i: j\in i}
e^{\lambda_i}\mu_{ji}^{-1}\Big) +
\sum_{i\in\mI} \lambda_in_i.
\end{align*} 
In the last inequality above, we let $p^{(j)}=(m_{ji}/m_j : i\ni j)$ and
$q^{(j)}=(e^{\lambda_i}\mu_{ji}^{-1}/\sum_r e^{\lambda_r}\mu_{jr}^{-1} : i\ni j)$. Recalling our Remark
\ref{relative entropy remark} on relative entropies, this Lagrangian is minimized by taking
$p^{(j)}=q^{(j)}$ for each $j\in\mJ$, where $D(p^{(j)}||q^{(j)}) =0$, 
and then by minimizing over $m_j$. In particular,  we get the Lagrange dual function
\begin{equation*}
\min_{m\in\bR_+^K} L(m,\lambda)=
\begin{cases}
\sum_{i: n_i>0} n_i\lambda_i &\text{if}\quad \sum_{i:j\in i} \frac{e^{\lambda_i}}{\mu_{ji}}  \leq
1,\;\;\forall j\in\mJ,\\
-\infty &\text{otherwise.}
\end{cases}
\end{equation*}
Thus, we find dual
\begin{equation*}
\text{maximize} \quad \sum_{i: n_i>0} n_i\lambda_i\quad \text{subject to} \quad
\sum_{i: j\in i} \frac{e^{\lambda_i}}{\mu_{ji}} \leq 1\quad\text{over}\quad \lambda\in\bR^I.
\end{equation*}
Substituting $\dL_i=e^{\lambda_i}$ gives the required result
\begin{equation*}
\text{maximize}\;\; \sum_{i: n_i>0} n_i\log {\dL_i}\quad
\text{subject to}\quad \sum_{i: j\in i} \frac{\dL_i}{\mu_{ji}} \leq 1, \quad \forall
j\in\mJ\quad\text{over}
\quad\dL\in\bR_+^I.
\end{equation*}
 \end{proof}

\begin{lemma}\label{flow bound}
 If, for some $\epsilon>0$, $\beta(m(t))\geq \beta(m^*)+\epsilon$ then there exists $\delta> 0$,
$i\in\mI$ and $j\in i$ such that
\begin{equation*}
 |\dL_{ji}(t)-\dL_{x_{ji}i}(t)|\geq \delta.
\end{equation*}
\end{lemma}
 \begin{proof}
We develop a proof by contradiction. If this result were not true, as the set of queue sizes is
compact, we could construct a sequence of times $t^k$, $k=1,2,3,...$ such that
$\beta(m(t^k))\geq \beta(m^*)+\epsilon$ and $m(t^k)\rightarrow \tilde{m}$, as $k\rightarrow \infty$
and 
\begin{equation*}
 \sum_{i\in\mI} \sum_{j\in i} \left| \dL_{ji}(t^k) - \dL_{ji}(t^k) \right|
\xrightarrow[k\rightarrow\infty]{} 0.
\end{equation*}
For each queue $j\in \mI$ with $\tilde{m}_j>0$, let $j^{+}_i$ be the next non-empty queue on route
$i$ i.e. $\tilde{m}_{j^{+}_i i}>0$. We can say
\begin{equation*}
 \dL_{ji}(t^k)\xrightarrow[k\rightarrow\infty]{}\frac{\tilde{m}_{ji}\mu_{ji}}{\tilde{m_j}}
\quad\text{and}\quad 
\dL_{j^{+}_ii}(t^k)\xrightarrow[k\rightarrow\infty]{}\frac{\tilde{m}_{j^{+}_ii}\mu_{j^{+}_ii}
} { \tilde{m}_{j^{+}}  }.
\end{equation*}
We letting $J^+$ be the set of queues on route $i$ between $j$ and $j^+_i$ that includes $j$ but
does not include $j^+_i$. Applying a triangle inequality across these queues, we can say that 
\begin{align*}
 \left| \frac{\tilde{m}_{ji}\mu_{ji}}{\tilde{m}_j} -\frac{\tilde{m}_{j^{+}_ii}\mu_{j^{+}_ii}
} { \tilde{m}_{j^{+}}}\right|  \leq & \lim_{k\rightarrow\infty} \left[  \left|
\frac{\tilde{m}_{ji}\mu_{ji}}{\tilde{m}_j}-\dL_{ji}(t^k) \right| + 
\left| \frac{\tilde{m}_{j^+_ii}\mu_{j^+_ii}}{\tilde{m}_{j^+_i}}-\dL_{j^+_ii}(t^k) \right|
+\sum_{l\in J^+ } \left| \dL_{li}(t^k) - \dL_{x_{li}i}(t^k)
\right| \right]=0.
\end{align*}
Applying this triangle inequality once more, for any queue $l$ on route $i$ that is between $j$ and
$j^+_i$, we see that
\begin{equation*}
\lim_{k\rightarrow \infty} \left| \frac{\tilde{m}_{ji}\mu_{ji}}{\tilde{m}_j} -
\dL_{li}(t^k)\right| =0.
\end{equation*}
In other words, for each route $i$ and for all queue $j\in i$, $\dL_{ji}(t^k)$ converges to some
value $\tilde{\aL}_i>0$ where if $\tilde{m}_j>0$ we have that
\begin{equation*}
\tilde{\aL}_i=\frac{\tilde{m}_{ji}\mu_{ji}}{\tilde{m}_j}
\end{equation*}
for some constant $\tilde{\aL}_i>0$. Observe that, by \eqref{fluidequns2}, for each queue $j\in
\mJ$
\begin{equation*}
 \sum_{i: j\in i} \frac{\tilde{\aL}_i}{\mu_{ji}} \leq 1
\end{equation*}
also
\begin{equation}\label{primal-dual}
 \beta(\tilde{m})=\sum_{i\in\mI} \sum_{j: j \in i} \tilde{m}_{ji} \log \tilde{\aL}_i =
\sum_{i\in\mI} n_i \log \tilde{\aL}_i. 
\end{equation}
Thus, the vector $\tilde{\aL}=(\tilde{\aL}_i: i\in \mI)$ is feasible for the dual problem
and the vector $\tilde{m}$ is feasible for the primal problem. We know by Weak Duality (for a
minimization) that any primal feasible solution is bigger than that of the dual. Thus, we know
by \eqref{primal-dual} that the primal equals dual solution and so $\tilde{m}$ must be optimal for
the primal problem i.e. $\beta(\tilde{m})=\beta(m^*)$. This must be a contradiction because by
assumption $\beta(m(t^k))\geq \beta(m^*)+\epsilon$ and thus by continuity of $\beta$,
$\beta(\tilde{m})\geq \beta(m^*)+\epsilon$.
 \end{proof}

We are now in a position to prove Theorem \ref{Thrm3}.

\begin{proof}[Proof of Theorem \ref{Thrm3}]
We found $\beta(m(t))$ was absolutely continuous in $t$. In Lemma \ref{formal diff bound}, we found
the derivative of $\beta(m(t))$ was almost everywhere negative and thus $\beta(m(t))$ must be a
decreasing function. 

Suppose for $s\in[0,t]$, $\beta(m(s)) \geq \beta(m^*)+\epsilon$ for some $\epsilon>0$ then by Lemma
\ref{flow bound} there exists an $i\in\mI$ and a $j\in i$ 
\begin{equation*}
\left| \dL_{ji}(s) - \dL_{ji}(s) \right| \geq \delta_{\epsilon}, 
\end{equation*}
for some  $\delta_{\epsilon}>0$. Thus, applying this to our bound in Lemma~\ref{formal diff bound} for intervals of time
$[0,t]$ such that $\beta(m(s)) \geq \beta(m^*)+\epsilon$, we have that
\begin{equation*}
\beta(m(t))\leq \beta(m(0)) - t\frac{\delta_{\epsilon}^2}{|\mJ|\mu_{max}}.
\end{equation*}
As $\beta(m(t))$ is bounded below by $\beta(m^*)$, the above inequality cannot be sustained for all
times $t$. In other words, eventually $\beta(m(t)) \leq \beta(m^*)+\epsilon$. Thus
$\beta(m(t))\searrow \beta(m^*)$. This proves the first assertion in Theorem \ref{Thrm3}. 

Now, it remains to show that $m(t)$ approaches $\mM$, the set of solutions to \eqref{optimize
beta}. Take some $\epsilon_1>0$. Let $m=(m_{ji}:i\in\mI, j\in i)$ be any vector with $\sum_{j:j\in
i} m_{ji}=n_i$ for $i\in\mI$ and such that 
\begin{equation*}
\min_{m^*\in\mathcal{M}} | m- m^*|\geq \epsilon_1.
\end{equation*}
Such an $m$ belongs to a compact set and thus, as $\beta$ is continuous, it must be that
$\beta(m)\geq \beta(m^*) +\epsilon$ for some $\epsilon>0$. Or stated differently, if $\beta(m)<
\beta(m^*) +\epsilon$ then it must be that 
\begin{equation*}
\min_{m^*\in\mathcal{M}} | m- m^*|< \epsilon_1.
\end{equation*}
As we have just shown, $\beta(m(t))< \beta(m^*) +\epsilon$ holds eventually for all fluid paths.
Thus, 
\begin{equation*}
\lim_{t\rightarrow\infty}\min_{m^*\in\mathcal{M}} | m(t)- m^*|=0.
\end{equation*}
 \end{proof}

 \appendix

\section{Proof of Lemma \ref{rate lemma}.}
\label{appendix1}


We consider both optimal solutions $\bar{\Lambda}^*(n)$ and $\Lambda^*(n)$. Let 
$$G_n(\Lambda)=\sum_{i\in\mI} n_i \log \Lambda_i$$
Since $\bar{\Lambda}^*(n)$ is the solution of an optimization with a larger feasible set
$G_n(\bar{\Lambda}^*(n))\geq G_n(\Lambda^*(n))$.

Take $v=\bar{\Lambda}^*(n)-\Lambda^*(n)$. Note $\Lambda^*(n)+\delta v$ belongs to feasible set
$$\bigg\{ \Lambda\geq 0 : \sum_{i: j\in i} \frac{\Lambda_i}{\mu_{ji}}\leq 1, \; j\in\mJ \bigg\}$$ 
for all $\delta$ suitably small. If this were not so then there would have been some constraint/queue
which we did not correctly include in the set of bottleneck links $\bar{J}$. Taking the partial
derivative of $G_n$ from $\Lambda^*(n)$ in the direction of $v$, we can then say that
\begin{equation*}
 \sum_{i\in\mI} v_i \frac{\partial
G_n(\Lambda^*(n))}{\partial \Lambda_i} \leq 0.
\end{equation*}
This holds because $\Lambda^*(n)$ is optimal. Now, also, by the concavity of $G_n(\cdot)$
\begin{equation*}
 G_n(\bar{\Lambda}^*(n))- G_n(\Lambda^*(n)) \leq  \sum_{i\in\mI} v_i \frac{\partial
G_n(\Lambda^*(n))}{\partial \Lambda_i}.
\end{equation*}
So, $G_n(\bar{\Lambda}^*(n))\leq G_n(\Lambda^*(n))$, and thus $G_n(\Lambda^*(n))=G_n(\bar{\Lambda}^*(n))$.
By the strict concavity the optimum of $G_n(\cdot)$ is unique, so, it must be that $\bar{\Lambda}^*(n)={\Lambda}^*(n)$.


\section{Lipschitz Continuity of $\beta(m(t))$.}\label{appendix2}
Before proceeding to prove the Lipschitz continuity of $\beta(m(t))$, i.e., Proposition~\ref{beta abs cont}, we give a proof
of Lemma~\ref{diff at zero}.

\begin{proof}[Proof of Lemma \ref{diff at zero}]
We use the fact that we know that the derivative exists. First, it is clear that $\frac{df}{dt}=0$
because
\begin{equation*}
 \frac{df}{dt}=\lim_{h\searrow 0} \frac{f(t+h)-0}{h} \geq 0\quad \text{and}  \quad
\frac{df}{dt}=\lim_{h\nearrow 0} \frac{f(t+h)-0}{h} \leq 0.
\end{equation*}
Now, noting that $f\log(p)$ is negative for all $p\leq 1$ and by using the same argument, we have
\begin{equation*}
 \frac{df\log p}{dt}=\lim_{h\searrow 0} \frac{f(t+h)\log p(t+h)-0}{h} \leq 0,\quad \text{and}  \quad
\frac{df\log p}{dt}=\lim_{h\nearrow 0} \frac{f(t+h)\log p(t+h)-0}{h} \geq 0.
\end{equation*}
Thus $\frac{df\log f}{dt}=0$.
  \end{proof}

We now demonstrate that the function $\beta(m(t))$ is Lipschitz continuous on any compact time
interval. Here $m(t)$ is any
solution to the fluid equations \eqref{fluidequns} and $\beta(m)$ is defined by \eqref{betalambda}.
The following arguments are adapted from Lemma 4.2 and Proposition 4.2 of Bramson \cite{Br96}. All
queues may empty in our network, so we have to apply some degree of care in proving the Lipschitz
continuity on compact time intervals.
First, we prove Lemma \ref{lamba diff}.

\begin{proof}[Proof of Lemma \ref{lamba diff}]
 Suppose that $m(t)$ and
$\aL(t)$ are differentiable at $t$. We may assume $\dL_{ji}(t)>0$ for some queue $j$ on
route $i$. Such a queue must exist because there is always some queue with $m_{ji}(t)>0$ as $n_i>0$
and thus by \eqref{fluidequns4} $\dL_{ji}(t)>0$. Now consider
$x_{ji}$ the next queue on route $i$, if $m_{x_{ji}i}(t)>0$ then by \eqref{fluidequns4}
$\dL_{x_{ji}i}>0$, and if $m_{x_{ji}i}(t)=0$ then $m'_{x_{ji}i}(t)=0$, thus
by \eqref{fluidequns1} $\dL_{x_{ji}i}(t)=\dL_{ji}(t)>0$. Continuing inductively we see
that $\dL_{ji}(t)>0$ for all queues.

From this argument we now see that the value of $\dL_{ji}(t)$ on any route $i$ is achieved by
a
queue $j^*$ with $m_{j^*i}>0$. Thus applying \eqref{fluidequns4}, $\dL_{ji}(t)\leq \mu_{ji}
\leq \mu_{max}$.
  \end{proof}

\begin{proposition}
For almost every $t>t_0>0$, there exists a constant $T>0$ such that if $t-t_0<\kappa T$ for
$\kappa\in\bN$ then
\begin{equation*}
\min_{ji} \dL_{ji}(t)> \frac{\min_{ji} \dL_{ji}(t_0)}{(1+\max_{ji} \mu_{ji})^\kappa}
\end{equation*}
Here $\kappa$ is a strictly positive constant which depends on $m(t_0)$, the fluid model state
at time $t_0$.
\end{proposition}

In order to prove this proposition we require the following lemma

\begin{lemma}\label{rate bound}
For almost every time $t_0$ and $t$ with $t>t_0$, if a queue $j$ on route $i$, has arrival
process from the queue before $j$, $l_{ji}$, such that for almost every 
$s\in [t_0,t]$
\begin{equation*}
 \dL_{l_{ij} j}(s) > c,
\end{equation*}
then the output of route $i$ from queue $j$ satisfies
\begin{equation}\label{Lambda below}
 \dL_{ji}(t) \geq c'
\end{equation}
where
\begin{equation*}
 c'=     \min\left\{ \dL_{ji}(t_0), \frac{c}{1+\max_{ji} \mu_{ji}} \right\}.
\end{equation*}
\end{lemma}

\begin{proof}

We note that $\aL_{ij}(t)$ is a Lipschitz function and thus is almost everywhere
differentiable. We assume that $t$ and $t'$ are differentiable points where
\eqref{Lambda below} is violated. 
Observe that if $m_{ji}(t)=0$ then, by
\eqref{fluidequns1}, $0=m'_{ji}(t)=\dL_{ji}(t)-\dL_{l_{ji}i}(t)$. Thus
$\dL_{ji}(t)=\dL_{l_{ji}i}(t)>c> c'$. So it must be that
$m_{ji}(t)>0$. Consequently by \eqref{fluidequns4}, $\dL_{ji}(s)=\mu_{ji}m_{ji}(s)/m_j(s)$ must
be continuous on an interval around $t$ and there must be an open interval around $t$ for which
$\dL_{ji}(s)<c'$. 

For a $\dL_{ji}(s)$ to get small we need the total number of departures
to be comparable relative to the arrivals. So, we will next argue the contradiction that
$\dL_{ji}(t)$ cannot enter an interval of time for which $\dL_{ji}(s)<c'$ without the
average departure rate $(\dL_{ji}(t)-\dL_{ji}(s))/(t-s)$ being bigger that $c'$.

We let $\tilde{t}$ be the last time before $t$
for which $\aL_{ji}(\tilde{t})\geq c'$. Note $\aL_{ji}(t_0)\geq c'$, so $\tilde{t}$ is well
defined. We use the shorthand $\aL_{ji}(\tilde{t},t)=\aL_{ji}(t)-\aL_{ji}(\tilde{t})$
and $\Delta_j(\tilde{t},t)=m_j(t)-m_j(\tilde{t})$. As $m_{ij}(t)>0$, by \eqref{fluidequns4}, we have
\begin{equation*}
c'> \dL_{ji}(t) = \mu_{ji}\frac{m_{ji}(t)}{m_j(t)}= \mu_{ji}\frac{ m_{ji}(\tilde{t}) +
\aL_{l_{ji}i}(\tilde{t},t) -\aL_{ji}(\tilde{t},t) }{m_j(\tilde{t})+ \Delta_j(\tilde{t},t)}.
\end{equation*}
Rearranging the above expression implies
\begin{equation*}
 \aL_{ji}(\tilde{t},t) > \aL_{l_{ji}i}(\tilde{t},t) - c' \Delta_j(\tilde{t},t) + \mu_{ji}
m_{ji}(\tilde{t}) - c'm_j(\tilde{t}) \geq \aL_{l_{ji}i}(t) - c' \Delta_j(\tilde{t},t).
\end{equation*}
The last inequality holds because, by \eqref{fluidequns4}, $\dL_{ji}(\tilde{t}) \geq c'$
implies $\mu_{ji}
m_{ji}(\tilde{t}) - c'm_j(\tilde{t})\geq 0$.
We now look at the mean value of the terms in the above inequality:
\begin{equation*}
 \frac{\aL_{ji}(\tilde{t},t)}{t-\tilde{t}}  > 
\frac{\aL_{l_{ji}i}(\tilde{t},t)}{t-\tilde{t}}  - c' \frac{\Delta_j(\tilde{t},t)}{t-\tilde{t}} 
\geq c - c' \max_{ji} \mu_{ji} \geq  c'.
\end{equation*}
In the second inequality above, we use the assumption that $\dL_{l_{ij} j}(s) > c,\; s\in
[t_0,t_0+T]$ and the fact the average change in $m_j(t)$, $\Delta_j(\tilde{t},t)/(t-\tilde{t})$, is
at most the maximum service rate $\max_{ji} \mu_{ji}$. The final inequality holds by our choice
of $c'$. 
This then contradicts our assumption that $ \dL_{ji}(s) < c' $ on the interval
$(\tilde{t},t]$.
  \end{proof}
The following lemma is a consequence achieved by iteratively applying the last result. We have to
apply some care because, in comparison to Bramson's open network analysis \cite{Br96}, every
queue can empty or have low arrival and departure rate.

\begin{lemma}\label{network lower bound}
For almost every $t_0$, there exists an interval of fixed length $[t_0,t_0+T]$ such that for almost
every $t\in[t_0,t_0+T]$ and for all $i\in\mI$ and $j\in\mJ$ 
\begin{equation*}
 \dL_{ji}(t)>\tilde{c}(t_0)>0.
\end{equation*}
Here $\tilde{c}(t_0)$ is a function of the network state at time $t_0$.
\end{lemma}
\begin{proof}
On route $i$ there is always one queue with greater than or equal to the average amount of work.
Without loss of generality, i.e. relabelling queues if necessary, we assume that this is the first
queue on route $i$. So, we have that for any route $i$
\begin{equation*}
 m_{j_i^1 i}(t_0) \geq \frac{\min_r n_r}{I}.
\end{equation*}
The biggest rate that this queue could decrease is $\max_{ji} \mu_{ji}$. So, we have
\begin{equation*}
 m_{j_i^1 i}(t) \geq \frac{\min_r n_r}{2I}\qquad \text{for}\quad t\in [t_0,t_0+T],
\end{equation*}
where we define $T=\frac{\min_r n_r}{2I\max_{ji} \mu_{ji}}$, and
\begin{equation*}
\aL_{ji}(t_0)=\mu_{ji}\frac{m_{ji}(t_0)}{m_j(t_0)} > \frac{\min_{ji}\mu_{ji} \min_r
n_r}{2I\sum_r n_r}=c_1.
\end{equation*}
We can now repeatedly apply Lemma \ref{rate bound}. Starting from $c_1$ above for $k=1,\ldots,k_i-1$
we define
\begin{equation*}
 c_{k+1}=  \min\left\{ \dL_{ji}(t_0), \frac{c_k}{1+\max_{ji} \mu_{ji}} \right\}.
\end{equation*}
Each time we iterate, we reduce $c_k$ by at least $1+\max_{ji} \mu_{ji}$. A simple lower bound, which
is sufficient for our purposes, is that 
\begin{equation*}
 c_k \geq \frac{\min_{ji} \dL_{ji}(t_0)}{(1+\max_{ji}\mu_{ji})^K}\bydef \tilde{c}(t_0)
\end{equation*}
where $K=\max_{i} k_i$ is the longest route within the queueing network. Thus,
from this repeated application of Lemma \ref{rate bound}, we have that
\begin{equation*}
 \dL_{ji}(t) \geq  \tilde{c}(t_0)
\end{equation*}
for almost every $t\in [t_0,t_0+T]$.
\end{proof}

In a similar manner to Proposition 4.2 of Bramson \cite{Br96}, now we show that $\beta(m(t))$ is
Lipschitz on any compact time interval.

\begin{proof}[Proof of Proposition \ref{beta abs cont}]
Note that without loss of generality, we may assume interval $[t_0,t]$ is of length less than or
equal to $T=\frac{\min_r n_r}{2I\max_{ji} \mu_{ji}}$, where $T$ was derived in the last lemma,
Lemma \ref{network lower bound}. If $t-t_0>T$ then we can split the interval $[t_0,t]$ in to
overlapping sub-interval of size $T$ and then use the largest Lipschitz constant found in each
sub-interval as a Lipschitz constant for $[t_0.t]$.

Note that 
\begin{equation*}
 \beta(m(t)) = \sum_{j\in\mJ} \sum_{\substack{i: j\in i\\m_{ji}(t)>0}} m_{ji}(t) \log
\frac{m_{ji}(t)\mu_{ji}(t)}{m_{j}(t)}=\sum_{j\in\mJ} \sum_{i: j\in i} m_{ji}(t) \log
\dL_{ji}(t).
\end{equation*}
It is enough to prove Lipschitz continuity of each term summed above:
$$
| m_{ij}(t_2) \log (\dL_{ji}(t_2)) 
- m_{ij}(t_1) \log (\dL_{ji}(t_1)) |
\leq 
D_1 \, | t_2 - t_1 |
\ .
$$
By Lemma \ref{network lower bound} for $s\in [t_0,t]$, $\log(\dL_{ji}(s))$ is bounded below by
$\tilde{c}(t_0)$ and also above by $m_j$, the maximum service rate of queue $j$. So
\begin{equation*}
| \log(\dL_{ji}(s))| \leq  D_0\bydef \max\{|\log(m_j)|, |\log( \tilde{c}(t_0))|\}.
\end{equation*}

If $m_{ij}(t_2) = m_{ij}(t_1) = 0$, the relation is trivial, so we assume that $m_{ij}(t_2)>0$.
If $m_{ij}(t_1) = 0$ we have that 
$$
| m_{ij}(t_2) \log (\dL_{ji}(t_2) ) | 
\leq D_0 | m_{ij}(t_2) | \,  
 = D_0 | m_{ij}(t_2) - m_{ij}(t_1) |
\leq D_2 \, D_0 \, | t_2 - t_1 |
$$
where the constant $D_0$ is as above and the constant $D_2$ is
the Lipschitz constant of $m_{ij}(t)$.

Now assume $m_{ij}(t_2)$ and $m_{ij}(t_1)$ both positive, and without loss of generality that
$$ \dL_{ji}(t_1) \leq  \dL_{ji}(t_2) \ .$$
It follows that
\begin{align*}
\left| m_{ij}(t_2) \log  \dL_{ji}(t_2) 
- m_{ij}(t_1)  \log  \dL_{ji}(t_1)  \right| 
 \leq & 
\left| m_{ij}(t_2) - m_{ij}(t_1) \right|  \, \left|\log  \dL_{ji}(t_2)  \right|\\
& + m_{ij}(t_1) \, \left| \log  \dL_{ji}(t_2)  -  \log
 \dL_{ji}(t_1)  \right|
\end{align*}
Again the first term in this upper bound is less than or equal to $D_2 \, D_0 \, | t_2 - t_1 |$.
The second term, since the logarithm function is concave and has its derivative maximized
at the left-most point, can be bounded in the following way
\begin{eqnarray*}
m_{ij}(t_1) \, \left| \log  \dL_{ji}(t_2)  -  \log
 \dL_{ji}(t_1)  \right|
&\leq&  \, \frac{m_{ij}(t_1)}{\dL_{ji}(t_1)} \left| \dL_{ji}(t_2)  
    -  \dL_{ji}(t_1) \right| \\
& = &\left| m_j(t_1) \frac{m_{ji}(t_2)}{m_j(t_2)} - m_{ij}(t_1) \right| \\
&\leq& \frac{m_{ij}(t_2)}{m_j(t_2)} \left|m_j(t_2) - m_j(t_1) \right|
+  \left|m_{ij}(t_2) - m_{ij}(t_1) \right| \\
&\leq &  \left|m_j(t_2) - m_j(t_1) \right|
+  \left|m_{ij}(t_2) - m_{ij}(t_1) \right| \\
& \leq & D_3 \, | t_2 - t_1 |,
\end{eqnarray*}
which follows from the fact that $m_{ij}(t_2) \leq m_j(t_2)$ and the fact that both $m_{ij}(t)$ and
$m_j(t)$ are Lipschitz continuous.
The result follows by choosing $D_1 \geq \max\{D_3 , D_2 \, D_0\}$.
  \end{proof}

\bibliographystyle{plain}
\bibliography{references}

\end{document}